\documentclass[review]{siamart}

\usepackage{lipsum}
\usepackage{amsfonts}
\usepackage{amssymb}
\usepackage{graphicx,color}
\usepackage{epstopdf}
\usepackage{algorithmic}
\usepackage{stmaryrd}
\usepackage{bookmark}
\usepackage{tikz}
\usetikzlibrary{decorations.pathreplacing}
\usetikzlibrary{arrows,calc,shapes,backgrounds}

\ifpdf
  \DeclareGraphicsExtensions{.eps,.pdf,.png,.jpg}
\else
  \DeclareGraphicsExtensions{.eps}
\fi

\newcommand{\TheTitle}{On the convergence to local limit of nonlocal models with approximated interaction neighborhoods}
\newcommand{\TheAuthors}{Qiang Du, Hehu Xie and Xiaobo Yin}

\title{{\TheTitle}}

\author{
  Qiang Du\thanks{Department of Applied Physics and Applied Mathematics, Columbia University, NY 10027, USA
    (\email{qd2125@columbia.edu}).
  The research of this author is supported in part by National Science Foundation grants DMS-2012562 and DMS-1937254.
   }
  \and
  Hehu Xie\thanks{LSEC, ICMSEC, Academy of Mathematics and Systems Science, Chinese Academy of Sciences, Beijing 100190, China, and  School of Mathematical Sciences, University of Chinese Academy of Sciences, Beijing, 100049, China (\email{hhxie@lsec.cc.ac.cn}). The research of this author is supported by the  Beijing Natural Science Foundation Z200003, National Natural Science Foundations of China 11771434 and the National Center for Mathematics and Interdisciplinary Science, CAS.}
  \and
  Xiaobo Yin\thanks{School of Mathematics and Statistics, and Hubei Key Laboratory of Mathematical Sciences, Central China Normal University, Wuhan 430079, China (\email{yinxb@mail.ccnu.edu.cn}). The research of this author is supported by the National Natural Science Foundation of China 11671165 and Hubei Provincial Science and Technology Innovation Base (Platform) Special Project 2020DFH002.}
}

\tikzset{
  flash/.style args={#1:#2}{postaction=decorate,decoration={name=markings,
    mark=at position #1 with {
    \draw[fill=#2, line width=.75\pgflinewidth, line cap=round, line join=round]
         (+\pgflinewidth,+7\pgflinewidth)   -- ++ ( left:+2\pgflinewidth)
      -- (+-4\pgflinewidth,+-\pgflinewidth) -- ++ (right:+5\pgflinewidth)
      -- (+-\pgflinewidth,+-7\pgflinewidth) -- ++ (right:+2\pgflinewidth)
      -- (+4\pgflinewidth,\pgflinewidth)    -- ++ (left:+5\pgflinewidth)
      -- cycle;}}}
      }

\ifpdf
\hypersetup{
  pdftitle={\TheTitle},
  pdfauthor={\TheAuthors}
}
\fi

\allowdisplaybreaks 

\newcommand{\Omdh}{{\widehat{\Omega}_{\delta}}}
\newcommand{\Omdc}{{\Omega_{\delta}^{c}}}

\begin{document}
\maketitle
\begin{abstract}
Many nonlocal models have adopted Euclidean balls as the nonlocal interaction neighborhoods. When solving them numerically, it is sometimes convenient to adopt polygonal approximations of such balls. A crucial question is, to what extent such approximations affect the nonlocal operators and the corresponding solutions. While recent works have analyzed this issue for a fixed horizon parameter, the question remains open in the case of a small or vanishing horizon parameter, which happens often in many practical applications and has significant impact on the reliability and robustness of nonlocal modeling and simulations.
In this work, we are interested in addressing this issue and establishing the convergence of the nonlocal solutions associated with polygonally approximated interaction neighborhoods to the local limit of the original nonlocal solutions. Our finding reveals that the new nonlocal solution does not converge to the correct local limit when the number of sides of polygons is uniformly bounded. On the other hand, if the number of sides tends to infinity, the desired convergence can be established. These results may be used to guide future computational studies of nonlocal models.
\end{abstract}



\section{Introduction}
Nonlocal models that account for interaction occurring at a distance have been shown to provide improved simulation fidelity in the presence of long-range forces and anomalous behaviors \cite{du2019nonlocal}. For this reason, they have found interesting applications in a variety of research fields such as fracture mechanics \cite{ha2011characteristics,littlewood2010simulation,silling2000reformulation}, phase transitions \cite{bates1999integrodifferential,delgoshaie2015non,fife2003some}, stochastic processes \cite{burch2014exit,d2017nonlocal,meerschaert2019stochastic,metzler2000random} and image processing \cite{buades2010image,gilboa2007nonlocal,gilboa2009nonlocal,lou2010image}.

There has been much recent interest in developing numerical algorithms for nonlocal
models \cite{d2020numerical}, including finite difference \cite{du2019asymptotically,tian2017conservative,tian2013analysis}, finite element \cite{du2019conforming,tian2013analysis,zhang2016quadrature}, collocation
\cite{tian2013efficient,wang2017fast,zhang2018accurate} and meshfree  \cite{lehoucq2016radial,leng2021asymptotically,leng2020asymptotically,silling2005meshfree} methods. For mesh dependent numerical schemes, like finite difference and finite element methods, the meshes are often composed of polyhedra in three dimension or polygons in the two dimensional case. On the other hand, the interaction neighborhoods are ubiquitously chosen to be Euclidean balls (the term `ball' is also used for circular disk in $\mathbb{R}^2$ here), leading to the challenge of dealing with intersections of such balls with the polygonal shapes. To mitigate the challenge, polygonal approximation of balls has been presented as a possible option \cite{d2021cookbook,vollmann2019nonlocal}. For mesh-free schemes, one also needs to take care of the numerical treatment of nonlocal interaction neighborhoods properly, often involving the implementation of various volume correction strategies, see \cite{parks2008peridynamics,zheng2021new}.
A natural question is, to what extent those approximations using polygonally approximated interaction neighborhoods affect the nonlocal operators and the corresponding nonlocal solutions. In particular, an important aspect of this question is how such approximations may affect the consistency with well-known local limits of the nonlocal models, should such limits make sense. This is the issue that we focus on in the current work in light of the potential significant impact on the reliability and robustness of nonlocal modeling and simulations.

We can reformulate the issue in terms of the diagram shown in \cref{fig:diagram}, which also summarizes the existing literature and the main contributions of this work. We use $\mathcal{L}_{\delta}$ to represent a nonlocal operator with the horizon parameter $\delta$ measuring the radius of the interaction neighborhood, the Euclidean ball in this paper. Meanwhile, $\mathcal{L}_{\delta,n_{\delta}}$ is used to represent
a nonlocal operator obtained from $\mathcal{L}_{\delta}$ by replacing the Euclidean ball with a suitable polygonal approximation. The integer $n_{\delta}$ denotes the maximum numbers of the polygons' sides, and we use $n_{\delta}=\infty$ to represent the case that the polygonal set with infinite number of sides recovers the Euclidean ball.
$\mathcal{L}_0 $ denotes the local limit of  $\mathcal{L}_{\delta}$ as $\delta\to 0$.  The nonlocal problem is stated as  $-\mathcal{L}_{\delta} u_{\delta} = f_{\delta} $ with $f_{\delta}$ denoting prescribed data and $u_{\delta}$ denoting the nonlocal solution. Similarly, we have the corresponding nonlocal problem $-\mathcal{L}_{\delta,n_{\delta}} u_{\delta,n_{\delta}} = f_{\delta} $ and local problem $-\mathcal{L}_0 u_{0} = f $ with $f$ be the limit of $f_{\delta} $ as $\delta\to 0$ in the $L^{2}$ sense (or in a weaker sense), respectively. The dashed and dotted arrows are the main directions studied in this paper, while the solid ones stand for cases previously investigated in the literature. In particular, the dashed line, marked with $\lightning$ signs, indicates a path that the nonlocal solution $u_{\delta,n_{\delta}}$ fails to converge to the correct local solution $u_0$ if $n_{\delta}$ remains bounded as $\delta\rightarrow 0$. This result can be found later in \cref{subsec:lossAC} where we show that for a generic quadratic function $q=q({\bf x})$, $\mathcal{L}_{\delta,n_{\delta}}q $ does not converge to $\mathcal{L}_{0}q$ if $n_{\delta}$ has a uniform upper bound as $\delta\to 0$.
On the other hand, the dotted line corresponds to a path that $u_{\delta,n_{\delta}}$ converges to $u_0$ if $n_{\delta}$ tends to infinity as $\delta\rightarrow 0$ under the condition that the polygonal approximation is a quasi-uniform family of inscribed polygons of the Euclidean balls, which is one of the three sufficient conditions presented in \cref{sec:condAC} to ensure the convergence of $u_{\delta,n_{\delta}}$ to $u_{0}$. The other two sufficient conditions are related to different ratios: one between the radius of the largest inscribed balls of the polygons and $\delta$, and another between areas of the polygons and the balls.
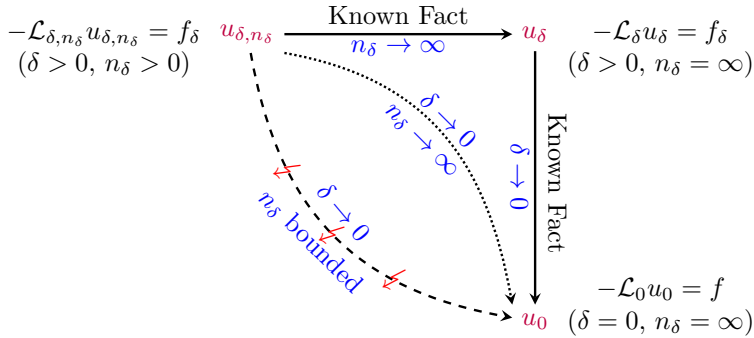
\begin{figure}[htb!]
\centering
\begin{tikzpicture}[scale=0.45]
\tikzset{to/.style={->,>=stealth,line width=.9pt}}

\node (v1) at (0,8.5)   {\textcolor{purple}{$u_{\delta,n_{\delta}}$}};
\node (v2) at (8.5,8.5) {\textcolor{purple}{$u_\delta$}};
\node (v4) at (8.5,0)   {\textcolor{purple}{$u_0$}};

\node[align=center, yshift=-0.2cm,anchor=east] at (v1.west)
  {$-\mathcal{L}_{\delta,n_{\delta}} u_{\delta,n_{\delta}} = f_{\delta} $\\[.01cm] ($\delta>0$, $n_{\delta}>0$)};

\node[align=center, yshift=0.2cm,anchor=west] at (v4.east)
  {$-\mathcal{L}_0 u_{0} = f $\\[.01cm] ($\delta=0$, $n_{\delta}=\infty$)};
\node[align=center,yshift=-0.2cm,anchor=west] at (v2.east)
  {$-\mathcal{L}_{\delta} u_{\delta} = f_{\delta} $\\[.01cm] ($\delta>0$, $n_{\delta}=\infty$)};

\draw[to] (v1.east) --  node[midway,below,yshift=0.05cm] {\textcolor{blue}{$n_{\delta}\to \infty$}}
node[midway,above] {\textcolor{black}{Known Fact}}
(v2.west);

\draw[to] (v2.south) --
node[midway,xshift=0.3cm, yshift=1.cm,rotate=-90, right]  {Known Fact}
node[midway,xshift=-0.2cm, yshift=0.6cm,rotate=-90, right]  {\textcolor{blue}{$\delta\to0$}}
(v4.north);

\draw[densely dotted][to] (v1.south east) to[out = -10, in = 100, looseness = 1.1] node[midway,xshift=-0.3cm,yshift=-0.1cm,rotate=-45]
{\textcolor{blue}{$n_{\delta}\rightarrow\infty$}}  node[midway,xshift=0.08cm,yshift=0.2cm,rotate=-45]    {\textcolor{blue}{$\delta\to0$}}
(v4.north west);

\draw[dashed][to] (v1) to[out = -80, in = -190, looseness = 1.1] node[midway,xshift=-0.3cm,yshift=-0.1cm,rotate=-45] {\textcolor{blue}{$n_{\delta}$ bounded}} node[midway,xshift=0.83cm,yshift=-0.6cm,rotate=-40]    {\textcolor{red}{\large\bf $\lightning$}}  node[midway,xshift=0.cm,yshift=0.0cm,rotate=-50]    {\textcolor{red}{\large\bf$\lightning$}} node[midway,xshift=-0.6cm,yshift=0.85cm,rotate=-60]   {\textcolor{red}{\large\bf $\lightning$}}
node[midway,xshift=0.15cm,yshift=0.3cm,rotate=-50]    {\textcolor{blue}{$\delta\to0$}}
(v4);

\end{tikzpicture}\caption{A convergence diagram representing nonlocal problems with different interaction neighborhoods and their local limits}
\label{fig:diagram}
\end{figure}

The convergence diagram in \cref{fig:diagram} is reminiscent to \cite[Figure 1]{tian2014asymptotically} (see also \cite[Figure 2.1]{tian2020asymptotically}) where the concept of Asymptotically Compatible (AC) schemes is formally proposed to study the convergence of numerical approximations to nonlocal problems and their local limits. One may find a number of  recent studies on various AC schemes for nonlocal models, see for example \cite{du2019asymptotically,du2019conforming,leng2022petrov,leng2021asymptotically,trask2019asymptotically,you2020asymptotically,zhang2018accurate} and additional references cited therein. While the diagram in \cref{fig:diagram} can be regarded as some kind of AC property, its implication is different. In fact, although motivated by computational studies, the diagram in \cref{fig:diagram} does not have to involve any numerical discretization. Instead, one may interpret the diagram completely on the continuum level so that the main focus is to examine the effect of the polygonal approximations of the spherical interaction neighborhood on  the nonlocal operators and solutions. Nevertheless, the findings given by the diagram shed light on the potential  risks involved in computational simulations of nonlocal problems with different polygonal approximations. They provide theoretical guidance to the design and analysis of numerical algorithms when approximations of interaction neighborhoods are adopted.

More precise statements and rigorous derivations of the findings shown in \cref{fig:diagram} are given later in \cref{sec:convdia}. In preparation, we present detailed descriptions on the  nonlocal operators $\mathcal{L}_{\delta}$ and $\mathcal{L}_{\delta,n_{\delta}}$, the local limit $\mathcal{L}_{0}$, as well as
some known results in the related literature in \cref{sec:notation}. In particular, the convergence of $u_{\delta}$ to $u_{0}$ is reviewed. We also provide some popular examples of kernels in \cref{subsec:kernel} to illustrate that  our theory is widely applicable. Verification of the theory, numerical results and discussions are
given in \cref{sec:numer} to validate the theoretical studies. Finally, we give some conclusions in \cref{sec:conclusion}.

\section{Background and notation}\label{sec:notation}
For $u({\bf x}): \mathbb{R}^{2}\rightarrow \mathbb{R}$, the nonlocal operator $\mathcal{L}_{\delta}$ acting on $u({\bf x})$ is defined as
\begin{equation}\label{Nolocal_Operator}
\mathcal{L}_{\delta}u({\bf x})=2\int_{\mathbb{R}^{2}}{\big (}u({\bf y})-u({\bf x}){\big )}\gamma_{\delta}({\bf x},{\bf y})d{\bf y}\quad \forall {\bf x}\in\Omega,
\end{equation}
where the nonnegative symmetric mapping $\gamma_{\delta}({\bf x},{\bf y}): \mathbb{R}^{2}\times\mathbb{R}^{2} \rightarrow \mathbb{R}$ is often called a kernel. In general, the integral in \cref{Nolocal_Operator} is understood in the principal value sense.
In this paper, we consider the following nonlocal homogeneous Dirichlet volume-constrained diffusion problem
\begin{equation}\label{nonlocal_diffusion}
     \left \{
     \begin{array}{rll}
-\mathcal{L}_{\delta}u_{\delta}({\bf x})&=f_{\delta}({\bf x})& \mbox{on}\: \Omega, \\
u_{\delta}({\bf x})&=0 & \mbox{on} \:\Omdc,
     \end{array}
     \right .
\end{equation}
where $\Omega\subset\mathbb{R}^{2}$ is a bounded open domain, the set $\Omdc$ is called the interaction region corresponding to the domain $\Omega$ and the kernel $\gamma_{\delta}$. It consists of those points in the complement domain $\mathbb{R}^{2}\!\setminus\!\Omega$ that interact with points in $\Omega$. $f_{\delta}({\bf x})\in L^2(\Omega)$ is a given function depending on $\delta$. Here we use homogeneous Dirichlet boundary condition to avoid the impact from boundary condition. Readers who are interested in the convergence behavior for nonlocal solutions to their local limits under different kinds of nonlocal boundary conditions are referred to \cite{cortazar2008approximate,du2019uniform,TTD17,yanguniform,you2020asymptotically}.  For convenience we denote by $\Omdh=\Omega\cup\Omdc$.

\subsection{Assumptions on the kernels}
In this work, we are interested in nonlocal models with the interactions occurring over finite distances. The kernels are thus assumed to have bounded supports, i.e., $\gamma_{\delta}({\bf x},{\bf y})\neq 0$ only if ${\bf y}$ is within a bounded neighborhood of ${\bf x}$, which represents the support of the kernel and is called the interaction neighborhood. Here we consider the choice of kernels with their support being Euclidean balls $B_{\delta,2}({\bf x})=\{{\bf y}\in\mathbb{R}^{2}: \|{\bf y}-{\bf x}\|_{2}\leq\delta\}$ (abbreviated as $B_{\delta}({\bf x})$), i.e.,
\begin{equation}\label{kernel_finite}
\forall {\bf x}\in\Omega:\:
\gamma_{\delta}({\bf x},{\bf y})=0,\: \forall {\bf y}\in\mathbb{R}^{2}\!\setminus\! B_{\delta}({\bf x}),
\end{equation}
where $\delta$ is known as the horizon parameter or interaction radius. Under this condition,
the interaction region becomes a layer around $\Omega$: $\Omdc=\{{\bf y}\in \mathbb{R}^{2}\!\setminus\!\Omega:\, \mbox{dist}({\bf y},\partial \Omega)<\delta\}$. Besides \cref{kernel_finite}, the kernel $\gamma_{\delta}({\bf x},{\bf y})$ is assumed to satisfy the following conditions:
\begin{equation}\label{General_kernel}
\left\{
\begin{array}{l}
\gamma_{\delta}({\bf x},{\bf y})=\widetilde{\gamma}_{\delta}(\|{\bf y}-{\bf x}\|_{2})>
 0\:\: \mbox{for}\:\: \|{\bf y}-{\bf x}\|_{2}< \delta, \;\text{ and }\\[.1cm]
 \displaystyle \int_{B_{\delta}({\bf 0})}z_{i}^{2}\widetilde{\gamma}_{\delta}(\|{\bf z}\|_{2})d{\bf z}=1,\:\; \text{ for }\;  i=1,2.
 \end{array}
 \right.
\end{equation}
Note that the first condition in \cref{General_kernel} means that the kernel is strictly positive for ${\bf y}$ inside $B_{\delta}({\bf x})$.  The second condition of \cref{General_kernel} is set to ensure that if the operator $\mathcal{L}_{\delta}$ converges to a limit as $\delta\to 0$, then the limit is given by $\mathcal{L}_{0}=\Delta$, the classical diffusion operator  \cite{tian2014asymptotically}.  Assuming further that $f_{\delta} \to f_0$ in $L^2(\Omega)$ as $\delta\to 0$, or a weaker assumption on the convergence \cite{tian2020asymptotically},  the corresponding local problem in the $\delta\to 0$ limit is given as follows:
\begin{equation}\label{local_diffusion}
     \left \{
     \begin{array}{rll}
-\mathcal{L}_{0}u_{0}({\bf x})&=f_{0}({\bf x})& \mbox{on}\: \Omega, \\
u_{0}({\bf x})&=0 & \mbox{on} \:\partial\Omega.
     \end{array}
     \right .
\end{equation}
This means that as $\delta\rightarrow 0$, the nonlocal effect diminishes and nonlocal equations converge to a classical local differential equation. Such limiting behavior reflects connections and consistencies between nonlocal and local models and has great practical significance especially for multiscale modeling and simulations.

To make the analysis later more concise, we further assume that $\widetilde{\gamma}_{\delta}$ has a re-scaled form, that is,
\begin{equation}\label{rescaled_kernel}
\widetilde{\gamma}_{\delta} ( \|{\bf y}-{\bf x}\|_{2})
=\frac{1}{\delta^{4}}\gamma\left(\frac{\|{\bf y}-{\bf x}\|_{2}}{\delta}\right)
\end{equation}
for some nonnegative function $\gamma$: $\mathbb{R}^{*}\rightarrow \mathbb{R}^{*}$
such that
\begin{equation*}
\int_{B_{1}({\bf 0})}\xi_{i}^{2}\gamma(\|{\boldsymbol \xi}\|_{2})d{\boldsymbol \xi}=1,\: i=1,2.
\end{equation*}
By \cref{kernel_finite} and \cref{General_kernel}, we see that the support set of $\widetilde{\gamma}_{\delta}$ is $B_{\delta}({\bf 0})$, the support set of $\gamma$ is $[0,1)$. Thus, for any sub-domain $D_{s}\subset B_{1}({\bf 0})$ of nonzero measure, we have
\begin{equation}\label{comp_supp}
\int_{D_{s}}\|{\boldsymbol \xi}\|_{2}^{2}\gamma(\|{\boldsymbol \xi}\|_{2})d{\boldsymbol \xi}>0.
\end{equation}

\subsection{Motivation for studying the polygonal approximations of the circular interaction neighborhoods}\label{subsec:Moti}
When solving nonlocal problems like \cref{nonlocal_diffusion} on computers using mesh based numerical methods such as finite element methods, a popular strategy is to use various polygonal approximations of circular interaction neighborhoods (Euclidean balls)  \cite{d2021cookbook,vollmann2019nonlocal}. This helps to eliminate the complications to handle nonlocal interactions at the ball-element intersections. To be specific, in (a) of \cref{Fig:Notconv} the Euclidean ball, centered at a point $O$, is approximated by an $8$-sides polygon. The difference between the ball and the polygon consists of eight circular caps (abbreviated as caps, the red colored domain is one of eight caps). A natural question is, to what extent the precision is affected when contributions from these caps are thrown away or approximated by some triangles. One can find in \cite{d2021cookbook,vollmann2019nonlocal} discussions on how the resulted approximations, in combination with quadrature rules, affect the discretization error, including the estimate on the convergence rate with respect to the mesh size $h$ between the original continuous linear finite element solution (denoted by $u_{\delta}^{h}$) associated with the exact Euclidean balls and that (denoted by $u_{\delta,n_{\delta}}^{h,\sharp}$) associated with polygonal approximations $B_{\delta,n_{\delta}}^{\sharp}$. Using suitably modified notations and descriptions, we review a known convergence result as follows.

\begin{proposition}\cite[Corollary 4.2]{d2021cookbook}
Assume that the kernel $\gamma_{\delta}({\bf x},{\bf y})$ is square integrable (or integrable and translationally invariant) and
bounded for all $B_{\delta}({\bf x})\ominus B_{\delta,n_{\delta}}^{\sharp}({\bf x})$ (which denotes the symmetric difference
between $B_{\delta}({\bf x})$ and $B_{\delta,n_{\delta}}^{\sharp}({\bf x})$),
where $\sharp$ represents one of the following possible choices
\begin{equation}\label{appro_ball}
\big\{\rm regular, nocaps, approxcaps, 3vertices, 123vertices, 23vertices, barycenter\big\}.
\end{equation}
Then, the following estimate holds
\begin{equation}\label{Error_h}
{\Big\|}u_{\delta}^{h}-u_{\delta,n_{\delta}}^{h,\sharp}{\Big\|}_{L^{2}}\leq CK\sup_{{\bf x}\in\Omega}{\Big|}B_{\delta}({\bf x})\ominus B_{\delta,n_{\delta}}^{\sharp}({\bf x}){\Big|}.
\end{equation}
\end{proposition}
Here $B_{\delta,n_{\delta}}^{\rm regular}({\bf x})$ is an $n_{\delta}$-sided inscribed regular polygon of $B_{\delta}({\bf x})$, while $B_{\delta,n_{\delta}}^{\rm nocaps}({\bf x})$ stands for the polygon generated by throwing away all caps formed whenever the circular boundary of the ball $B_{\delta}({\bf x})$ intersects the sides of a triangle. The caps are approximated by sub-triangles which, together with $B_{\delta,n_{\delta}}^{\rm nocaps}({\bf x})$, results to $B_{\delta,n_{\delta}}^{\rm approxcaps}({\bf x})$. $B_{\delta,n_{\delta}}^{\sharp}\!({\bf x})$, with
$\sharp$ being one of $\{\rm 3vertices, 123vertices, 23vertices, barycenter\}$, stands for the cases that `three vertices', `at least one vertex', `at least two vertices', or `barycenter' of the triangle is inside the ball, respectively. These seven types of polygonal approximations are depicted in \cref{Fig:Polygon_Approx}. For a more detailed description, please refer to \cite{d2021cookbook}.

Based on the geometric estimate,
\begin{equation*}
{\Big|}B_{\delta}({\bf x})\ominus B_{\delta,n_{\delta}}^{\sharp}({\bf x}){\Big|}=\mathcal{O}(h^{\alpha}),\:\mbox{with}\:\alpha\in [1,2],
\end{equation*}
it was shown that
\begin{equation}\label{Error_gap_order}
{\Big\|}u_{\delta}^{h}-u_{\delta,n_{\delta}}^{h,\sharp}{\Big\|}_{L^2}\sim\mathcal{O}(h^{\alpha}),
\end{equation}
when the continuous linear finite element approximation is used, for which the rate of convergence is $\mathcal{O}(h^2)$ at best. For example, with $\sharp={\rm nocaps}$, $\alpha=2$ is shown in \cite{d2021cookbook}.
\begin{figure}[tbhp]
\centering
\vspace{-1cm}
\captionsetup[subfigure]{captionskip=-18pt}
\subfloat[\rm regular]{\includegraphics[width=4.5cm]{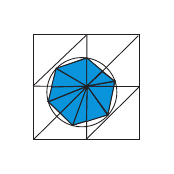}}\hspace{-1cm}
\subfloat[\rm nocaps]{\includegraphics[width=4.5cm]{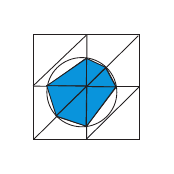}}\hspace{-1cm}
\subfloat[\rm approxcaps]{\includegraphics[width=4.5cm]{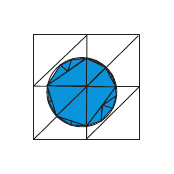}}\\
\vspace{-0.5cm}
\captionsetup[subfigure]{captionskip=2pt}
\subfloat[\rm 3vertices]{\includegraphics[width=3cm]{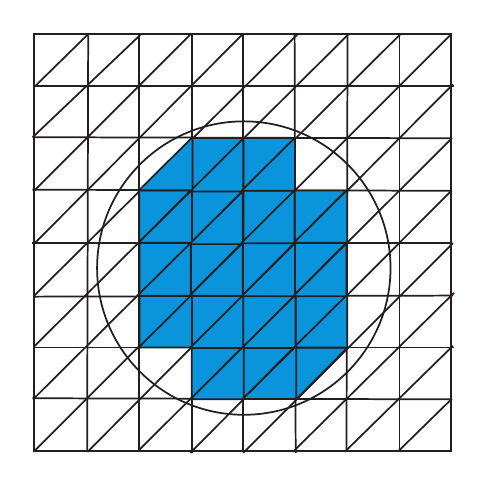}}
\subfloat[\rm 123vertices]{\includegraphics[width=3cm]{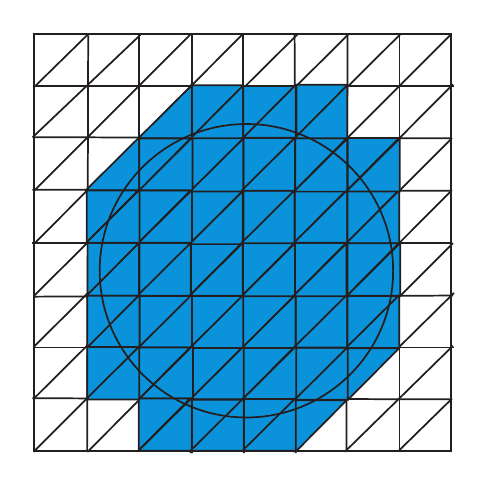}}
\subfloat[\rm 23vertices]{\includegraphics[width=3cm]{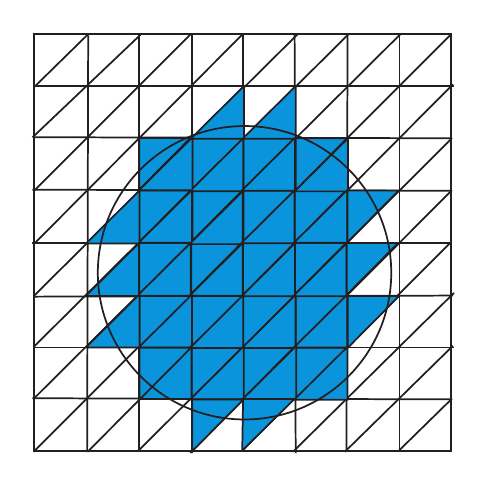}}
\subfloat[\rm barycenter]{\includegraphics[width=3cm]{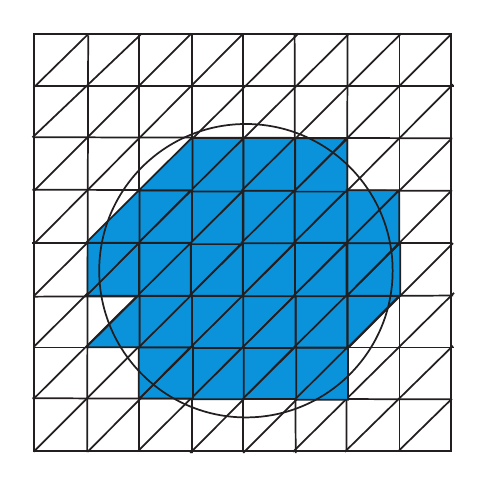}}
\caption{Polygonally approximated balls}\label{Fig:Polygon_Approx}
\end{figure}

While the existing convergence results are encouraging, we note that the estimate \cref{Error_h} is derived by ignoring the dependence on the parameter $\delta$.
Indeed, the development in \cite{d2021cookbook} was for the case of a finite fixed $\delta >0$ and for a specific mesh.
A careful examination of the assumption on the kernel for achieving \cref{Error_h} and the detailed proof in \cite{d2021cookbook} indicates that \cref{Error_h}  does not hold uniformly as $\delta\to 0$. In fact, to reflect the dependence on $\delta$, the constant $K$ in \cref{Error_h} needs to be modified as $K(\gamma_{\delta})$, which is defined by
\begin{equation}\label{K_gamma}
K(\gamma_{\delta})=\sup_{{\bf x}\in\Omega}\left(\int_{\left(B_{\delta}({\bf x})\ominus B_{\delta,n_{\delta}}^{\sharp}({\bf x})\right)\cap\Omdh}\gamma_{\delta}({\bf x},{\bf y})d{\bf y}\right).
\end{equation}
For the constant kernel considered in \cite[Section 8]{d2021cookbook}, that is
\begin{equation}\label{const_2d}
\gamma(t)=\frac{4}{\pi},\:\: \mbox{for}\:\: 0\leq t< 1,
\end{equation}
one obtains $K(\gamma_{\delta})\sim\mathcal{O}(\delta^{-4})$ according to \cref{K_gamma}. Thus \cref{Error_gap_order} becomes
\begin{equation}\label{diverge}
{\Big\|}u_{\delta}^{h}-u_{\delta,n_{\delta}}^{h,\sharp}{\Big\|}_{L^2}\sim\mathcal{O}(\delta^{-4}h^{\alpha}),\:\mbox{with}\:\alpha\in [1,2],
\end{equation}
which means that one may not always expect the convergence for any pair of $(\delta,h)$ as both $\delta\rightarrow 0$ and $h\rightarrow 0$. As suggested in \cite{silling2005meshfree}, the ratio $\delta/h=3$ is usually adopted in practice for macroscale modeling by peridynamics model which is a popular class of nonlocal models. With this choice of the horizon parameter and mesh parameter, the convergence likely might not hold according to \cref{diverge}.
This motivates the current study. However, to avoid complications, we choose to first examine the effect of polygonal approximations of Euclidean balls on the continuum level in this work, while leaving the study that also accounts for the numerical discretization to subsequent investigations. Indeed, our approach can also be adopted in the discrete setting so that similar results can also be established for numerical
approximations. For example, if continuous linear finite element method is used, the question raised earlier could be answered afresh: for all types of polygonal approximations of Euclidean balls listed in \cref{appro_ball}, if the ratio $\delta/h$ is uniformly bounded, ${\big\|}u_{\delta}^{h}-u_{\delta,n_{\delta}}^{h,\sharp}{\big\|}_{L^2}$ may not converge to zero as $\delta\rightarrow 0$.

Note that in the theory developed in this paper, the operator $\mathcal{L}_{\delta,n_{\delta}}$ is not assumed to depend on the mesh used for numerical discretization. However, if such a mesh is provided at hand, then it could be used to construct polygons, like those types of polygons in \cref{Fig:Polygon_Approx}.

\subsection{Nonlocal operators with polygonal approximations of
circular interaction neighborhoods}\label{sec:nonl}
We aim to study the convergence property for a series of continuum nonlocal operators $\mathcal{L}_{\delta,n_{\delta}}$ as $\delta\to 0$, which depends on a family of polygons that approximate the circular interaction neighborhoods (Euclidean balls).

As in \cite{du2012analysis} the nonlocal energy norm and nonlocal volume constrained energy space are defined by
\begin{equation*}
\| u\|_{\delta} :=\left(\int_{\Omdh}\int_{\Omdh}{\big(}u({\bf y})-u({\bf x}){\big)}^2\gamma_{\delta}({\bf x},{\bf y})d{\bf y}d{\bf x}\right)^{1/2},
\end{equation*}
and
\begin{equation*}
V(\Omdh):={\Big\{}u\in L^2(\Omdh): \| u\|_{\delta}<\infty, u({\bf x})=0\:\mbox{on} \:\Omdc{\Big\}},
\end{equation*}
respectively.
\begin{figure}[tbhp]
\centering
\subfloat[convex polygon]{\includegraphics[width=4cm]{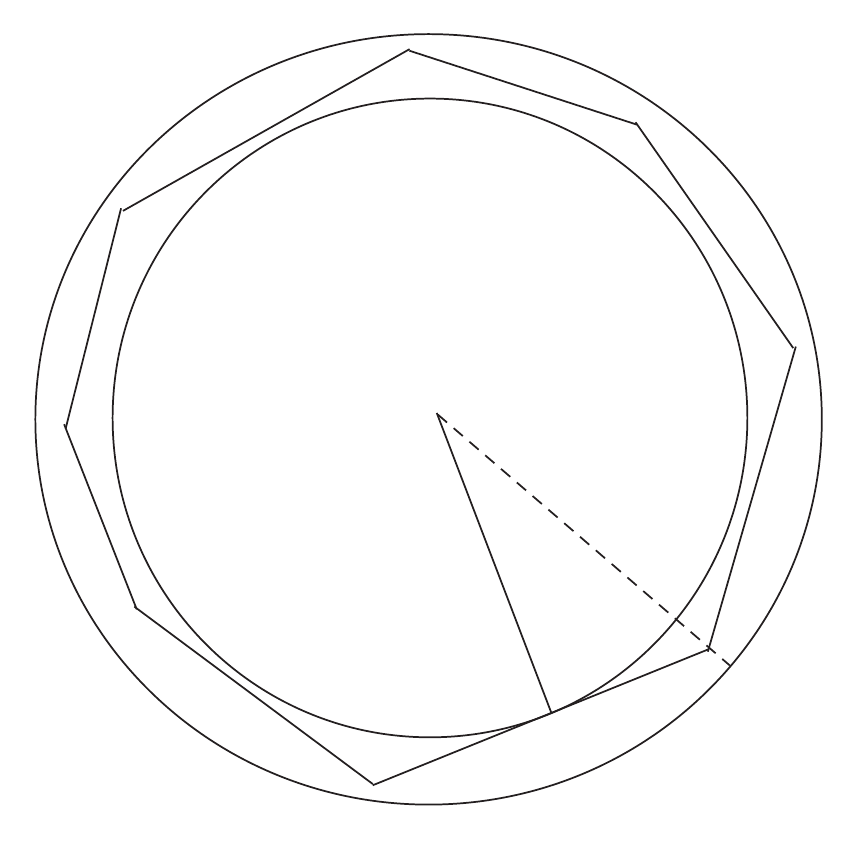}\put(-62,57){${\bf x}$}
\put(-61,29){$r_{\delta,{\bf x}}$}
\put(-32,38){$\delta$}}\hspace{1cm}
\subfloat[non-convex polygon]{\includegraphics[width=4cm]{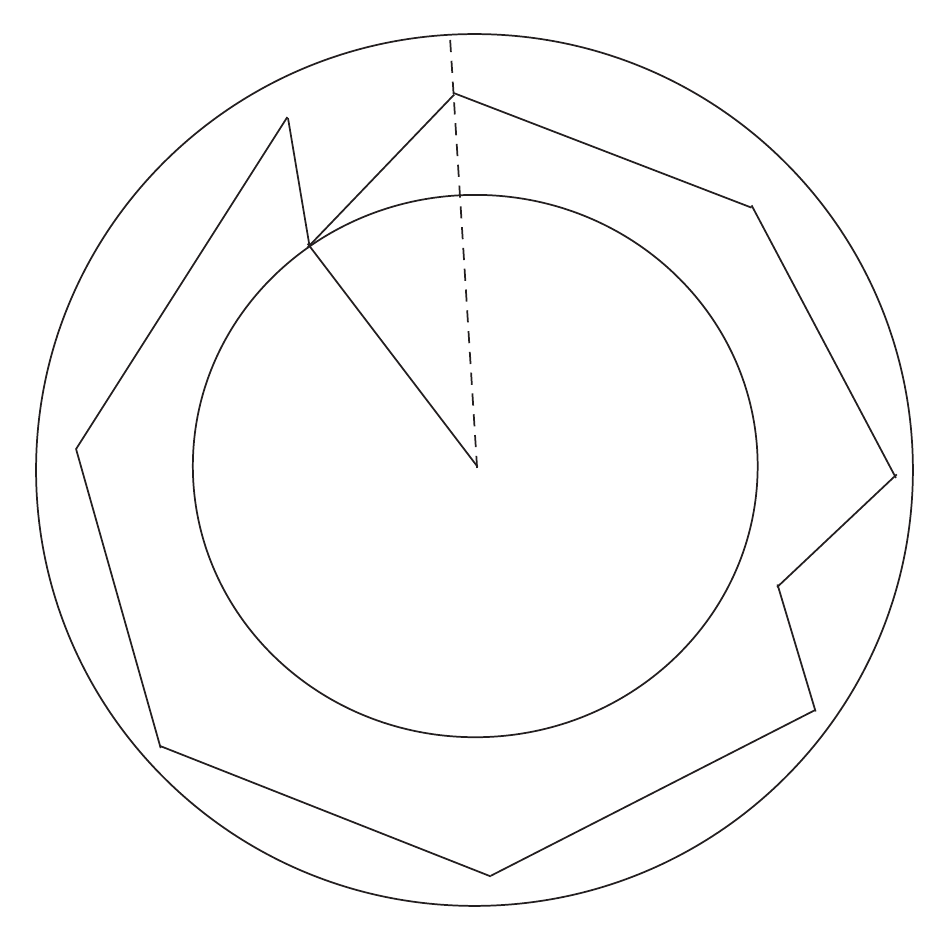}\put(-60,51){${\bf x}$}
\put(-80,65){$r_{\delta,{\bf x}}$}
\put(-56,72){$\delta$}}
\caption{Polygons in the ball $B_{\delta}({\bf x})$}\label{Fig:Polygon}
\end{figure}
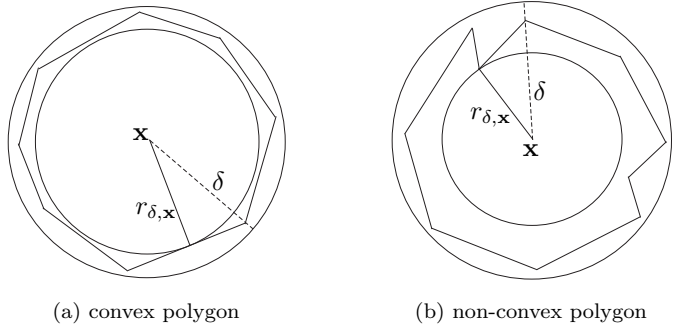
For any ${\bf x}\in\Omega$, a polygon inside the ball $B_{\delta}({\bf x})$ is denoted by $B_{\delta,n_{\delta,{\bf x}}}(\bf x)$, or simply $B_{\delta,n_{\delta,{\bf x}}}$, where $n_{\delta,{\bf x}}$ denotes the number of sides for the polygon. Furthermore, to make use of \cref{rescaled_kernel}, the shifted and re-scaled polygon of $B_{\delta,n_{\delta,{\bf x}}}$ is defined as
\begin{equation}\label{Affine_map}
B_{1,n_{\delta,{\bf x}}}({\bf 0})={\big\{}{\bf z}=\left({\bf y}-{\bf x}\right)/\delta:{\bf y}\in B_{\delta,n_{\delta,{\bf x}}}{\big\}}.
\end{equation}
For a given $\delta>0$, we introduce the notation
\begin{equation}
n_{\delta}=\sup_{{\bf x}\in\Omega}n_{\delta,{\bf x}},\quad\text{and}\quad n_{\delta,\inf}=\inf_{{\bf x}\in\Omega}n_{\delta,{\bf x}}.
\end{equation}
Meanwhile, let $r_{\delta,{\bf x}}$ (see \cref{Fig:Polygon}) be the radius of the largest inscribed ball (centered at ${\bf x}$) of $B_{\delta,n_{\delta,{\bf x}}}$, and
\begin{equation}\label{under_delta}
\underline{r}(n_{\delta})=\inf_{{\bf x}\in\Omega}r_{\delta,{\bf x}}.
\end{equation}

Set
\begin{equation*}
\gamma_{\delta,n_{\delta,{\bf x}}}({\bf x},{\bf y}) =\left \{
     \begin{array}{ll}
\gamma_{\delta}({\bf x},{\bf y}), & {\bf y}\in B_{\delta,n_{\delta,{\bf x}}},\\
0, & {\bf y}\notin B_{\delta,n_{\delta,{\bf x}}},
     \end{array}
     \right.
\end{equation*}
which is not always symmetric with respect to ${\bf x}$ and ${\bf y}$ since ${\bf y}\in B_{\delta,n_{\delta,{\bf x}}}$ is not always equivalent to ${\bf x}\in B_{\delta,n_{\delta,{\bf y}}}$. In order to inherit the symmetry of the original kernel $\gamma_{\delta}$, which is crucial to make the nonlocal Green's first identity \cite{du2013nonlocal} valid, we define a new family of kernels
\begin{equation}\label{kernel}
\gamma_{\delta,n_{\delta}}({\bf x},{\bf y})=\frac{\gamma_{\delta,n_{\delta,{\bf x}}}({\bf x},{\bf y})+\gamma_{\delta,n_{\delta,{\bf y}}}({\bf y},{\bf x})}{2},
\end{equation}
which are symmetric with respect to ${\bf x}$ and ${\bf y}$ but are not radial, i.e., not a function of $\|{\bf x}-{\bf y}\|_2$. Note that there is a special family of ${\big\{} B_{\delta,n_{\delta,{\bf x}}}{\big\}}$ such that the corresponding $\gamma_{\delta,n_{\delta,{\bf x}}}({\bf x},{\bf y})$ are symmetric with respect to ${\bf x}$ and ${\bf y}$ if $B_{\delta,n_{\delta,{\bf x}}}$ is defined as the same inscribed regular even-sided polygon of $B_{\delta}({\bf x})$ for any ${\bf x}$. Here, `the same' polygon means one having the identical shape and orientation, but shifted to be around ${\bf x}$. The symmetry is guaranteed by the fact that an inscribed regular even-sided polygon is symmetric about the point ${\bf x}$. We denote this special family of ${\big\{} B_{\delta,n_{\delta,{\bf x}}}{\big\}}$ as ${\big\{} B_{\delta,n}{\big\}}$.

Based on the definition of $\gamma_{\delta,n_{\delta}}$, a new family of nonlocal operators and the corresponding class of  nonlocal problems are defined as
\begin{equation}\label{disc_nonlocal_oper}
\mathcal{L}_{\delta,n_{\delta}}u({\bf x})=2\int_{\mathbb{R}^{2}}{\big (}u({\bf y})-u({\bf x}){\big )}\gamma_{\delta,n_{\delta}}({\bf x},{\bf y})d{\bf y}\quad \forall {\bf x}\in\Omega,
\end{equation}
and
\begin{equation}\label{nonlocal_approx}
     \left \{
     \begin{array}{rll}
-\mathcal{L}_{\delta,n_{\delta}}u_{\delta,n_{\delta}}({\bf x})&=f_{\delta}({\bf x})& \mbox{on}\: \Omega, \\
u_{\delta,n_{\delta}}({\bf x})&=0 & \mbox{on} \:\Omdc.
     \end{array}
     \right .
\end{equation}
The nonlocal energy norm associated with $-\mathcal{L}_{\delta,n_{\delta}}$ is defined as follows
\begin{equation*}
\|u({\bf x})\|_{\delta,n_{\delta}}=\left(\int_{\Omdh}\int_{\Omdh}{\big (}u({\bf y})-u({\bf x}){\big )}^{2}\gamma_{\delta,n_{\delta}}({\bf x},{\bf y})d{\bf y}d{\bf x}\right)^{1/2}.
\end{equation*}
By changing the roles of variables ${\bf x}$ and ${\bf y}$, it holds that
\begin{equation*}
\int_{\Omdh}\int_{\Omdh}{\big (}u({\bf y})-u({\bf x}){\big )}^{2}\gamma_{\delta,n_{\delta,{\bf x}}}({\bf x},{\bf y})d{\bf y}d{\bf x}=\int_{\Omdh}\int_{\Omdh}{\big (}u({\bf y})-u({\bf x}){\big )}^{2}\gamma_{\delta,n_{\delta,{\bf y}}}({\bf y},{\bf x})d{\bf y}d{\bf x},
\end{equation*}
then
\begin{align*}
\|u({\bf x})\|_{\delta,n_{\delta}}^{2}&=\int_{\Omdh}\int_{\Omdh}{\big (}u({\bf y})-u({\bf x}){\big )}^{2}\gamma_{\delta,n_{\delta,{\bf x}}}({\bf x},{\bf y})d{\bf y}d{\bf x}\\
&=\int_{\Omdh}\int_{\Omdh\cap B_{\delta,n_{\delta,{\bf x}}}}{\big (}u({\bf y})-u({\bf x}){\big )}^{2}\gamma_{\delta}({\bf x},{\bf y})d{\bf y}d{\bf x}.
\end{align*}

\section{Discussion on the convergence diagram}\label{sec:convdia}
To borrow the notion introduced in \cite{tian2014asymptotically,tian2020asymptotically}, if  $u_{\delta,n_{\delta}}$ converges to $u_{0}$ for both dashed and dotted arrows in \cref{fig:diagram},
the polygonal approximation is said to possess the AC property. In other words, AC property ensures $u_{\delta,n_{\delta}}$ be convergent to $u_{0}$ as $\delta\rightarrow 0$ no matter $n_{\delta}$ stays bounded or not.

\subsection{Local limit of nonlocal solutions of models with circular interaction neighborhoods}\label{subsec:Limit_Euclid_ball}
Assume that $\{\delta_{k}\}_{k=1}^{\infty}$ is a decreasing sequence and $\delta_{k}\rightarrow 0$, while $u_{\delta_{k}}$ is the solution of the nonlocal problem \cref{nonlocal_diffusion} with $\delta=\delta_{k}$. From \cite{du2013nonlocal,mengesha2014bond,mengesha2015variational,tian2014asymptotically,tian2020asymptotically} we know that if $f_{\delta_{k}} \rightarrow f_{0}$ in $L^{2}$ or in a weaker sense, then $u_{\delta_{k}} \rightarrow u_{0}$, while $u_{0}\in H^{1}_{0}(\Omega)$ satisfies
\begin{equation}\label{local_weak}
\int_{\Omega}\nabla u_{0}({\bf x})\nabla\varphi({\bf x}) d{\bf x}=\int_{\Omega}f({\bf x})\varphi({\bf x})d{\bf x},\: \forall\varphi\in \mathcal{D}(\Omega),
\end{equation}
which is the weak form of the local problem \cref{local_diffusion}. This result corresponds to the solid arrow in \cref{fig:diagram} from $u_{\delta}$ to $u_{0}$. We next analyze the convergence property to the local solution $u_{0}$ for $u_{\delta_{k},n_{\delta_{k}}}$ which solves the nonlocal problem \cref{nonlocal_approx} with $\delta=\delta_{k}$.

\subsection{Loss of the AC property with approximate polygonal interaction neighborhoods}\label{subsec:lossAC}
For $q({\bf x})=\|{\bf x}\|_{2}^{2}$, simple calculation leads to
\begin{equation*}
\mathcal{L}_{\delta}q({\bf x})=2\int_{B_{\delta}({\bf 0})}\|{\bf z}\|_{2}^{2}\widetilde{\gamma}_{\delta}(\|{\bf z}\|_{2})d{\bf z}
=4,\quad \mathcal{L}_{0}q({\bf x})
=4.
\end{equation*}
On the other hand, by the symmetry of $\gamma_{\delta,n_{\delta}}$, for all ${\bf x}\in \Omega$,
\begin{align}\label{lower_b}
&\mathcal{L}_{\delta,n_{\delta}}q({\bf x})
=2\int_{B_{\delta}({\bf x})}{\big (}q({\bf y})-q({\bf x}){\big )}\gamma_{\delta,n_{\delta}}({\bf x},{\bf y})d{\bf y}\nonumber\\
&=2\int_{B_{\delta}({\bf x})}\|{\bf y}-{\bf x}\|_{2}^{2}\gamma_{\delta,n_{\delta}}({\bf x},{\bf y})d{\bf y}+4{\bf x}\cdot\int_{B_{\delta}({\bf x})}({\bf y}-{\bf x})\gamma_{\delta,n_{\delta}}({\bf x},{\bf y})d{\bf y}\nonumber\\
&=2\int_{B_{\delta}({\bf x})}\|{\bf y}-{\bf x}\|_{2}^{2}\gamma_{\delta,n_{\delta}}({\bf x},{\bf y})d{\bf y}\nonumber\\
&=\int_{B_{\delta}({\bf x})}\|{\bf y}-{\bf x}\|_{2}^{2}\gamma_{\delta,n_{\delta,{\bf x}}}({\bf x},{\bf y})d{\bf y}+\int_{B_{\delta}({\bf x})}\|{\bf y}-{\bf x}\|_{2}^{2}\gamma_{\delta,n_{\delta,{\bf y}}}({\bf y},{\bf x})d{\bf y}\nonumber\\
&\leq\int_{B_{\delta,n_{\delta,{\bf x}}}}\|{\bf y}-{\bf x}\|_{2}^{2}\gamma_{\delta}({\bf x},{\bf y})d{\bf y}+\int_{B_{\delta}({\bf x})}\|{\bf y}-{\bf x}\|_{2}^{2}\gamma_{\delta}({\bf y},{\bf x})d{\bf y}\nonumber\\
&=\int_{B_{1,n_{\delta,{\bf x}}}({\bf 0})}\|\boldsymbol \xi\|_{2}^{2}\gamma(\|{\boldsymbol \xi}\|_{2})d{\boldsymbol \xi}+2.
\end{align}

Now we show that $\mathcal{L}_{\delta,n_{\delta}}q({\bf x})$ does not converge to $\mathcal{L}_{0}q({\bf x})$ when $n_{\delta}$ is uniformly bounded with $\delta$, as $\delta\rightarrow 0$. This means the polygonal approximation does not possess AC property.
\begin{theorem}\label{thm:Notconv}
Assume that ${\big\{} B_{\delta,n_{\delta,{\bf x}}}{\big\}}$ is a family of polygons that satisfy
\begin{equation}\label{Subset}
B_{\delta,n_{\delta,{\bf x}}} \subset B_{\delta}({\bf x}),\: \forall \delta,\: \forall {\bf x} \in \Omega.
\end{equation}
The family of kernels $\{\gamma_{\delta}\}$ satisfies the conditions \cref{kernel_finite} and \cref{General_kernel}. If $n_{\delta}$ is uniformly bounded as $\delta\rightarrow 0$, then $\mathcal{L}_{\delta,n_{\delta}}q({\bf x})$ does not converge to $\mathcal{L}_{0}q({\bf x})$ for any ${\bf x} \in \Omega$.
\end{theorem}
\begin{proof}
Since $n_{\delta}$ is uniformly bounded as $\delta\rightarrow 0$, there exists an integer $N$ such that $n_{\delta}\leq N$ for all $\delta$. We prove the theorem in three steps.

The first step, the conclusion is proven under the condition that ${\big\{} B_{\delta,n_{\delta,{\bf x}}}{\big\}}$ is a family of inscribed polygons of $B_{\delta}({\bf x})$. Here `inscribed' means all vertices of the polygon lie on the boundary of $B_{\delta}({\bf x})$. In this case the central angle of the longest side of the polygon $B_{1,n_{\delta,{\bf x}}}({\bf 0})$, denoted by $\theta_{n_{\delta,{\bf x}}}$, must be greater than or equal to $2\pi/n_{\delta}$. We consider the sector corresponding to the longest side, which is illustrated in (a) of \cref{Fig:Notconv} by the grey triangle and its abutting red cap. The red cap is denoted by $\widehat{C}_{n_{\delta,{\bf x}}}$, then $\widehat{C}_{n_{\delta,{\bf x}}}\subset B_{1}({\bf 0})\setminus B_{1,n_{\delta,{\bf x}}}({\bf 0})$ and
\begin{equation*}
{\big|}\widehat{C}_{n_{\delta,{\bf x}}}{\big|}=\frac{1}{2}\!\left(\theta_{n_{\delta,{\bf x}}}-\sin(\theta_{n_{\delta,{\bf x}}})\right)\geq \frac{1}{2}\!\left(\frac{2\pi}{n_{\delta}}-\sin\left(\frac{2\pi}{n_{\delta}}\right)\right)\geq \frac{1}{2}\!\left(\frac{2\pi}{N}-\sin\left(\frac{2\pi}{N}\right)\right),
\end{equation*}
which, together with \cref{lower_b}, leads to
\begin{align*}
\mathcal{L}_{0}q({\bf x})-\mathcal{L}_{\delta,n_{\delta}}q({\bf x})&\geq 2-\int_{B_{1,n_{\delta,{\bf x}}}({\bf 0})}\|\boldsymbol \xi\|_{2}^{2}\gamma(\|{\boldsymbol \xi}\|_{2})d{\boldsymbol \xi}\nonumber\\
&\geq \int_{B_{1}({\bf 0})\setminus B_{1,n_{\delta,{\bf x}}}({\bf 0})}\|\boldsymbol \xi\|_{2}^{2}\gamma(\|{\boldsymbol \xi}\|_{2})d{\boldsymbol \xi}
\geq \int_{\widehat{C}_{n_{\delta,{\bf x}}}}\|\boldsymbol \xi\|_{2}^{2}\gamma(\|{\boldsymbol \xi}\|_{2})d{\boldsymbol \xi}.
\end{align*}
Then by the condition \cref{comp_supp}, the conclusion is proven in this case.

As the second step, we consider the case that ${\big\{} B_{\delta,n_{\delta,{\bf x}}}{\big\}}$ is a family of convex polygons of $B_{\delta}({\bf x})$. There exists a family of inscribed polygons of $B_{\delta}({\bf x})$, which contains ${\big\{} B_{\delta,n_{\delta,{\bf x}}}{\big\}}$ and has the same numbers of sides with it, see (b) of \cref{Fig:Notconv}. The conclusion then follows from  the derivation given in the previous step.

Finally, as the third step, we deal with the case that ${\big\{} B_{\delta,n_{\delta,{\bf x}}}{\big\}}$ is a family of non-convex polygons of $B_{\delta}({\bf x})$. There exists a family of convex polygons of $B_{\delta}({\bf x})$, which contains ${\big\{} B_{\delta,n_{\delta,{\bf x}}}{\big\}}$ and has less numbers of sides than the latter, see (c) of \cref{Fig:Notconv}. Thus, the conclusion follows again from the previous step.
\end{proof}
\begin{figure}[tbhp]
\centering
\subfloat[inscribed polygon]{\includegraphics[width=4cm]{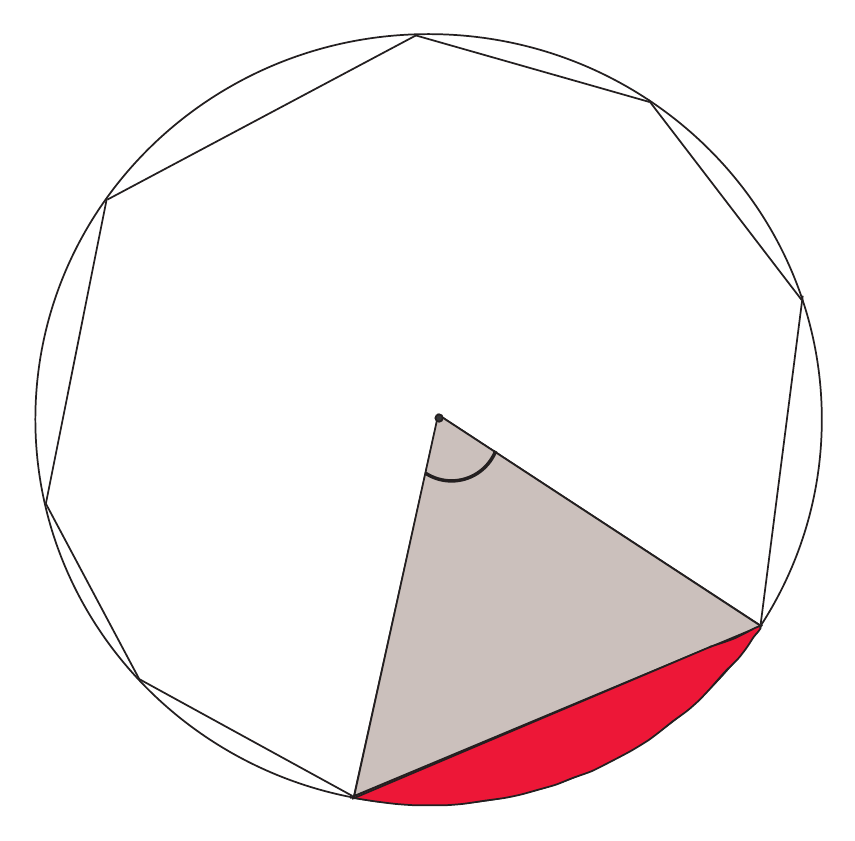}\put(-65,53){$O$}
\put(-54,40){$\theta_{n_{\delta,{\bf x}}}$}}
\hspace{0.2cm}
\subfloat[convex polygon]{\includegraphics[width=4cm]{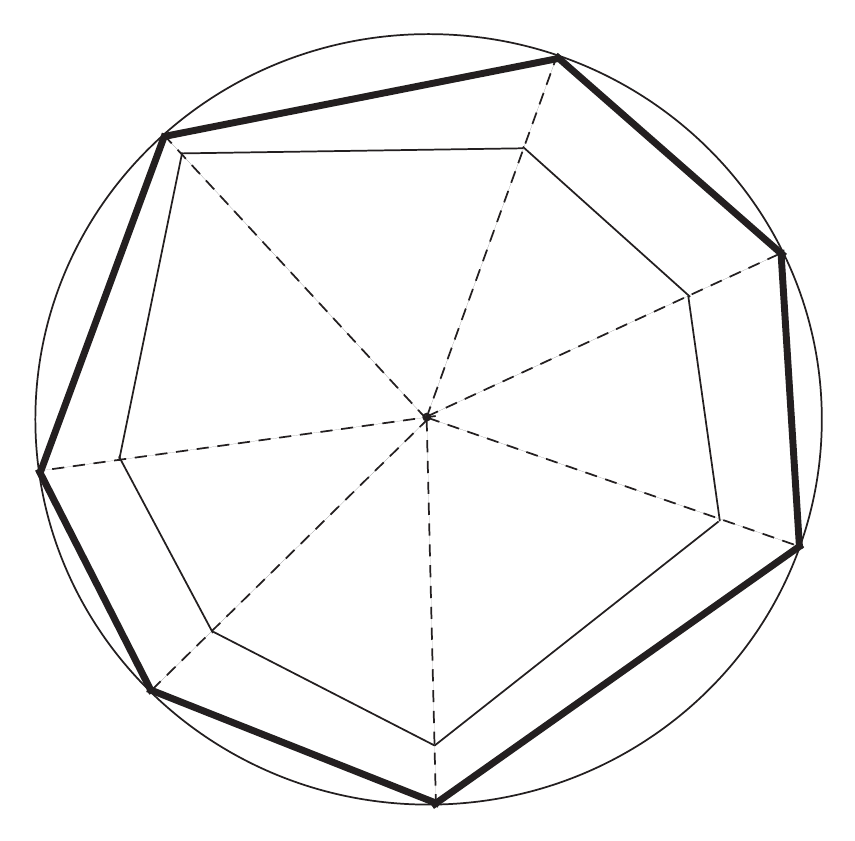}\put(-56,47){$O$}}\hspace{0.2cm}
\subfloat[non-convex polygon]{\includegraphics[width=4cm]{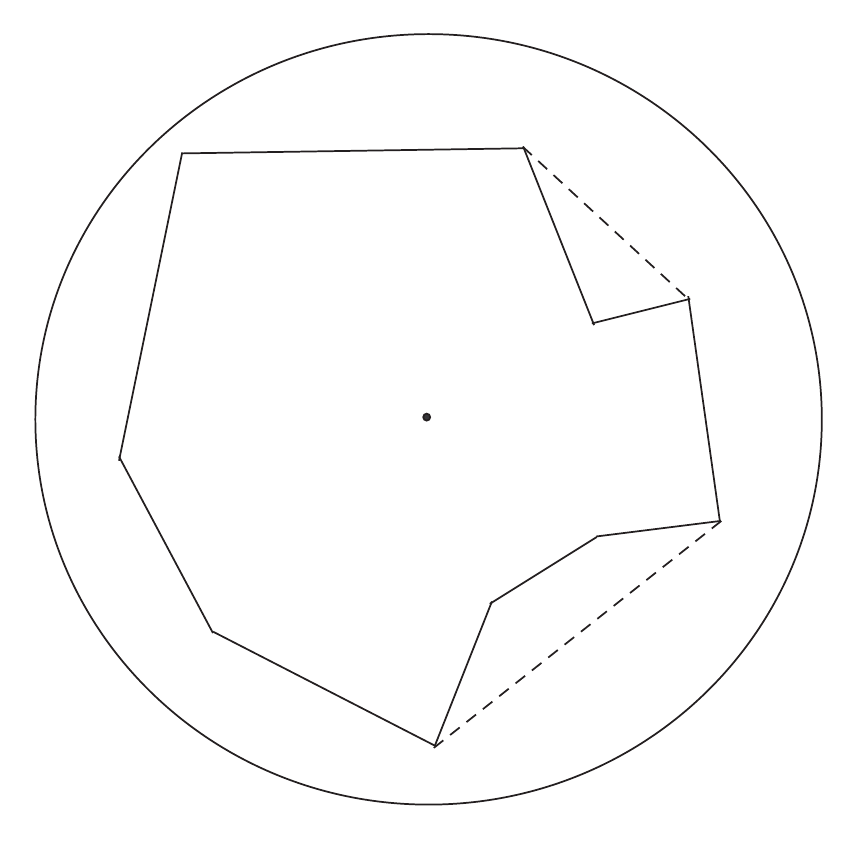}\put(-67,51){$O$}}
\caption{Different types of polygons in the ball $B_{1}({\bf 0})$}\label{Fig:Notconv}
\end{figure}
Now we consider a special case where the kernel is selected as \cref{const_2d}, while a family of inscribed regular polygons ${\big\{} B_{\delta,n}{\big\}}$ ($n\geq 4$ is a given even number) defined after \cref{kernel} is adopted. For this special case,
$\gamma_{\delta,n_{\delta}}({\bf x},{\bf y})=\gamma_{\delta,n_{\delta,{\bf x}}}({\bf x},{\bf y})$. Thus $\forall {\bf x}\in\Omega$,
\begin{equation}\label{conv_wrong}
\mathcal{L}_{\delta,n_{\delta}}q({\bf x})=2\int_{B_{1,n_{\delta,{\bf x}}}({\bf 0})}\!\|\boldsymbol \xi\|_{2}^{2}\gamma(\|{\boldsymbol \xi}\|_{2})d{\boldsymbol \xi}=C_{n}<4=\mathcal{L}_{0}q({\bf x}),
\end{equation}
with
\begin{equation*}
C_{n}=\frac{\sin(2\pi/n)}{\pi/2n}\cdot\frac{2+\cos(2\pi/n)}{3}.
\end{equation*}
This means, for a fixed $n$, $\mathcal{L}_{\delta,n_{\delta}}q({\bf x})$ is a constant $C_n/4$ multiple of $\mathcal{L}_{0}q({\bf x})$, rather than $\mathcal{L}_{0}q({\bf x})$ itself. Thus, we do not have the convergence of
$\mathcal{L}_{\delta,n_{\delta}}q({\bf x})$ to $\mathcal{L}_{0}q({\bf x})$ as $\delta\to 0$, though a re-scaled version does give a consistent limit since
 \begin{equation}\label{conv_scale}
\frac{4}{C_n}\mathcal{L}_{\delta,n_{\delta}}q({\bf x})=\mathcal{L}_{0}q({\bf x}).
\end{equation}

Now, we observe that while for any fixed $n$, $C_n/4<1$, it holds that
$C_n/4 \to 1$ as $n\to \infty$. In fact,
if we could prove
\begin{equation}\label{oper_cond}
\lim\limits_{\delta\rightarrow 0, n_\delta\to \infty}{\big(}\mathcal{L}_{\delta,n_{\delta}}-\mathcal{L}_{0}{\big)}q({\bf x})=0,\:\forall {\bf x}\in\Omega.
\end{equation}
Then $u_{\delta,n_{\delta}}$ may converge to $u_{0}$ along the dotted arrow in \cref{fig:diagram} which means the polygonal approximation possesses conditional AC property. We now discuss this convergence issue in  subsequent subsections.

\subsection{Conditional AC property with approximate polygonal interaction neighborhoods}\label{sec:condAC}
\subsubsection{Local-nonlocal Green's formula}\label{subsec:Conv_2D}
As the starting point of the convergence analysis of the local limit, let us first establish an estimate that is referred as a local-nonlocal Green's formula.
Indeed, it is derived from the nonlocal Green's formula where part of the integrals are expanded in local forms. As $\delta\to 0$, it recovers the classical Green's formula.
\begin{lemma}\label{lemma:Green_Nonlocal} (local-nonlocal Green's formula). We have for all $ u\in V(\Omdh)$ and $\varphi\in \mathcal{D}(\Omega)$,
\begin{equation}\label{Green_Nonlocal}
\int_{\Omega}\!\varphi({\bf x})\mathcal{L}_{\delta,n_{\delta}}u({\bf x})d{\bf x}
=\int_{\Omega}u({\bf x})\sum_{i=1}^{2}\sigma_{n_{\delta},i}({\bf x})\partial_{ii}^{2}\varphi({\bf x})d{\bf x}+\mathcal{O}{\big(}\delta^2{\big)}|\varphi|_{4,\infty}\|u\|_{L^1(\Omega)},
\end{equation}
with
\begin{equation}\label{def:sigma}
\sigma_{n_{\delta},i}({\bf x})=\int_{B_{\delta}({\bf x})}(y_{i}-x_{i})^{2}\gamma_{\delta,n_{\delta}}({\bf x},{\bf y})d{\bf y},\; i=1,2.
\end{equation}
\end{lemma}
\begin{proof}
Due to the symmetry of the kernel $\gamma_{\delta,n_{\delta}}({\bf x},{\bf y})$ with respect to ${\bf x}$ and ${\bf y}$, we have for all $u\in V(\Omdh)$ and $\varphi\in \mathcal{D}(\Omega)$,
\begin{align*}
&\int_{\Omega}\varphi({\bf x})\mathcal{L}_{\delta,n_{\delta}}u({\bf x})d{\bf x}=2\int_{\Omega}\varphi({\bf x})\int_{\Omdh}{\big(}u({\bf y})-u({\bf x}){\big)}\gamma_{\delta,n_{\delta}}({\bf x},{\bf y})d{\bf y}d{\bf x}\\
&=2\int_{\Omdh}\varphi({\bf x})\int_{\Omdh}{\big(}u({\bf y})-u({\bf x}){\big)}\gamma_{\delta,n_{\delta}}({\bf x},{\bf y})d{\bf y}d{\bf x}\\
&=2\int_{\Omdh}u({\bf x})\int_{\Omdh}{\big(}\varphi({\bf y})-\varphi({\bf x}){\big)}\gamma_{\delta,n_{\delta}}({\bf x},{\bf y})d{\bf y}d{\bf x}
=\int_{\Omega}u({\bf x})\mathcal{L}_{\delta,n_{\delta}}\varphi({\bf x})d{\bf x},
\end{align*}
which is in fact a kind of the nonlocal Green's first identity \cite{du2013nonlocal}, however, for the newly defined nonlocal operator $\mathcal{L}_{\delta,n_{\delta}}$,  rather than $\mathcal{L}_{\delta}$.
Since $\varphi\in \mathcal{D}(\Omega)$ is smooth, we expand the difference $\varphi({\bf y})-\varphi({\bf x})$ by the Taylor polynomial and proceed as follows
\begin{equation*}
\mathcal{L}_{\delta,n_{\delta}}\varphi({\bf x})=2\int_{\Omdh}\left(
\sum_{|{\boldsymbol{\alpha}}|\in \mathbb{N}^{+}}\frac{({\bf y}-{\bf x})^{\boldsymbol{\alpha}}}{{\boldsymbol{\alpha}}!}d^{{\boldsymbol{\alpha}}}\varphi({\bf x})\right)\gamma_{\delta,n_{\delta}}({\bf x},{\bf y})d{\bf y}.
\end{equation*}
Since all polynomials $({\bf y}-{\bf x})^{\boldsymbol{\alpha}}$ for odd $|\boldsymbol{\alpha}|$ are odd functions, their products with the kernel $\gamma_{\delta,n_{\delta}}({\bf x},{\bf y})$ vanish identically under the integral. Thus
\begin{align*}
&\mathcal{L}_{\delta,n_{\delta}}\varphi({\bf x})=2\int_{\Omdh}\left(
\sum_{|{\boldsymbol{\alpha}}|=2}\frac{({\bf y}-{\bf x})^{\boldsymbol{\alpha}}}{{\boldsymbol{\alpha}}!}d^{{\boldsymbol{\alpha}}}\varphi({\bf x})+\sum_{|{\boldsymbol{\alpha}}|= 4}\frac{({\bf y}-{\bf x})^{\boldsymbol{\alpha}}}{{\boldsymbol{\alpha}}!}d^{{\boldsymbol{\alpha}}}\varphi(\boldsymbol{\xi})\right)\gamma_{\delta,n_{\delta}}({\bf x},{\bf y})d{\bf y}\\
&=\sum_{i=1}^{2}\partial_{ii}^{2}\varphi({\bf x})\!\int_{\Omdh}(y_{i}-x_{i})^{2}\gamma_{\delta,n_{\delta}}({\bf x},{\bf y})d{\bf y}+2\!\sum_{|{\boldsymbol{\alpha}}|= 4}\!\int_{\Omdh}\!d^{{\boldsymbol{\alpha}}}\varphi(\boldsymbol{\xi})\gamma_{\delta,n_{\delta}}({\bf x},{\bf y})
\frac{({\bf y}-{\bf x})^{\boldsymbol{\alpha}}}{{\boldsymbol{\alpha}}!}d{\bf y},
\end{align*}
where $\boldsymbol{\xi}$ depends on $\bf y$ and $\|\boldsymbol{\xi}-{\bf x}\|_{2}\leq\|{\bf y}-{\bf x}\|_{2}$.
Next, the second term is shown to be of high order. Only the following two kinds of sub-terms need to be dealt with:
\begin{equation*}
\int_{\Omdh}\frac{\partial^{4}\varphi(\boldsymbol{\xi})}{\partial x_{1}^{4}}(y_{1}-x_{1})^{4}\gamma_{\delta,n_{\delta}}({\bf x},{\bf y})
d{\bf y},\:\:\mbox{and}\:\:
\int_{\Omdh}\frac{\partial^{4}\varphi(\boldsymbol{\xi})}{\partial x_{1}^{2}\partial x_{2}^{2}}(y_{1}-x_{1})^{2}(y_{2}-x_{2})^{2}\gamma_{\delta,n_{\delta}}({\bf x},{\bf y})
d{\bf y}.
\end{equation*}
For the sub-term of the first kind, we have by the condition \cref{General_kernel}
\begin{align*}
&\int_{\Omdh}\frac{\partial^{4}\varphi(\boldsymbol{\xi})}{\partial x_{1}^{4}}(y_{1}-x_{1})^{4}\gamma_{\delta,n_{\delta}}({\bf x},{\bf y})
d{\bf y}\leq\delta^{2}|\varphi|_{4,\infty}\int_{\Omdh}(y_{1}-x_{1})^{2}\gamma_{\delta,n_{\delta}}({\bf x},{\bf y})
d{\bf y}\\
&\leq \delta^{2}|\varphi|_{4,\infty}\int_{\Omdh}(y_{1}-x_{1})^{2}\gamma_{\delta}({\bf x},{\bf y})
d{\bf y}=\delta^{2}|\varphi|_{4,\infty}.
\end{align*}
The sub-term of the second kind could be dealt with similarly. Thus,
\begin{equation*}
\mathcal{L}_{\delta,n_{\delta}}\varphi({\bf x})=\sum_{i=1}^{2}\partial_{ii}^{2}\varphi({\bf x})\int_{\Omdh}(y_{i}-x_{i})^{2}\gamma_{\delta,n_{\delta}}({\bf x},{\bf y})d{\bf y}+\mathcal{O}{\big(}\delta^2{\big)}|\varphi|_{4,\infty},
\end{equation*}
we then complete the proof by the definition of $\sigma_{n_{\delta},i}({\bf x})$ in \cref{def:sigma}.
\end{proof}

By the definition of $\underline{r}(n_{\delta})$ in \cref{under_delta}, we have for all $u\in V(\Omdh)$,
\begin{align}\label{equiv}
&\int_{\Omdh}\int_{\Omdh}{\big|}u({\bf y})-u({\bf x}){\big|}^2\gamma_{\delta}({\bf x},{\bf y})\chi_{\underline{r}(n_{\delta})}(\|{\bf y}-{\bf x}\|_{2})d{\bf y}d{\bf x}\\
&\leq \int_{\Omdh}\int_{\Omdh}{\big|}u({\bf y})-u({\bf x}){\big|}^2\gamma_{\delta,n_{\delta}}({\bf x},{\bf y})d{\bf y}d{\bf x}\nonumber\\
&\leq \int_{\Omdh}\int_{\Omdh}{\big|}u({\bf y})-u({\bf x}){\big|}^2\gamma_{\delta}({\bf x},{\bf y})d{\bf y}d{\bf x}=\|u({\bf x})\|_{\delta}^{2}.\nonumber
\end{align}
We regard $\gamma_{\delta}({\bf x},{\bf y})\chi_{\underline{r}(n_{\delta})}(\|{\bf y}-{\bf x}\|_{2})$ as a new kernel, which, like $\gamma_{\delta}({\bf x},{\bf y})$, induces a norm
\begin{equation*}
\|u({\bf x})\|_{\underline{r}(n_{\delta})}=\left(\int_{\Omdh}\int_{\Omdh}{\big|}u({\bf y})-u({\bf x}){\big|}^2\gamma_{\delta}({\bf x},{\bf y})\chi_{\underline{r}(n_{\delta})}(\|{\bf y}-{\bf x}\|_{2})d{\bf y}d{\bf x}\right)^{1/2}.
\end{equation*}
Then from \cref{equiv} we know that for all $u\in V(\Omdh)$
\begin{equation}\label{equiv_norm}
\|u({\bf x})\|_{\underline{r}(n_{\delta})}\leq\|u({\bf x})\|_{\delta,n_{\delta}}\leq\|u({\bf x})\|_{\delta},
\end{equation}
which in fact holds for all $u\in L^{2}(\Omdh)$ with $u({\bf x})=0$ for ${\bf x}\in\Omdc$, if terms in \cref{equiv_norm}
are allowed to be infinite. Like \cite{bourgain2001another,brezis2002recognize}, we prove the corresponding pointwise convergence for $\|u\|_{\delta,n_{\delta}}^{2}$ in general $d$ dimensional case
under the condition $u\in H^{1}(\Omdh)$ which is a homogeneous extension from a function in $H_0^1(\Omega)$, see \cref{sec:appendix}.

\subsubsection{Sufficient conditions to ensure correct local limit}\label{subsec:condlimit}
Now, we are ready to show the $L^{2}$-convergence of $u_{\delta,n_{\delta}}$ to $u_{0}$ as $\delta\rightarrow 0$. For this purpose we need
to provide a compactness result and a uniform nonlocal Poincar\'{e}-type
inequality which enable us to carry on the same convergence proof as in
\cite[Theorem 2.5]{tian2014asymptotically}. For brevity we use $D$ to stand for $\widehat{\Omega}_{\delta_{0}}$ with $\delta\leq\delta_{0}$ for all $\delta$ or $\delta_{k}\leq\delta_{0}$ for all $k$, which depends on the situation. This notation is also used in \cref{sec:appendix}.

Set $\rho_{k}({\bf x},{\bf y})=\|{\bf y}-{\bf x}\|_{2}^2\gamma_{\delta_{k}}({\bf x},{\bf y})$. Since $\gamma_{\delta_{k}}$ satisfy \cref{kernel_finite} and \cref{General_kernel}, then $\{\rho_{k}\}$ is a sequence of radial functions and the following property holds for all ${\bf x}\in\Omega$
\begin{equation}\label{rho_cond_Stand}
     \left \{
     \begin{array}{l}
\rho_{k}({\bf x},{\bf y})\geq 0\:\: \mbox{in} \:\:\mathbb{R}^{2},\\
\displaystyle
\int_{\mathbb{R}^{2}}\rho_{k}({\bf x},{\bf y})d{\bf y}=2,\: \forall k\geq 1,\\
\displaystyle
\lim\limits_{k\rightarrow\infty}\int_{\|{\bf y}-{\bf x}\|_{2}>\varepsilon}\rho_{k}({\bf x},{\bf y})d{\bf y}=0,\: \forall \varepsilon>0.
     \end{array}
     \right.
\end{equation}
Thus the following compactness lemma for the norm $\|\cdot\|_{\delta}$ could be proven, see, e.g.
\cite[Theorem 4]{bourgain2001another}, \cite[Theorems 1.2, 1.3]{Ponce2004An} and \cite[theorem 5.1]{mengesha2012nonlocal}.
\begin{lemma}\label{lemma:Compact_Stand}
\cite{bourgain2001another,mengesha2012nonlocal,Ponce2004An} Assume $\{u_{k}\}\subset L^{1}(D)$ is a bounded sequence such that
\begin{equation*}
\int_{D}\int_{D}{\big|}u_{k}({\bf y})-u_{k}({\bf x}){\big|}^2\gamma_{\delta_{k}}({\bf x},{\bf y})d{\bf y}d{\bf x}\leq C,
\end{equation*}
then $\{u_{k}\}$ is precompact in $L^{2}(D)$. Assume that $u_{k_{j}}\rightarrow u$ in $L^{2}(D)$, then $u\in H^{1}(D)$.
\end{lemma}
In addition we need the following compactness lemma which is similar to \cref{lemma:Compact_Stand}, however, the kernel $\gamma_{\delta_{k}}$ is replaced by $\gamma_{\delta_{k},n_{\delta_{k}}}$.
\begin{lemma}\label{lemma:Compact_Modify}
Assume that $\{u_{k}\}\subset L^{1}(D)$ is a bounded sequence and
\begin{equation}\label{General_Cond_2D}
\lim\limits_{k\rightarrow\infty}\int_{\|{\bf y}-{\bf x}\|_{2}\leq\underline{r}(n_{\delta_{k}})}(y_{i}-x_{i})^2\gamma_{\delta_{k}}({\bf x},{\bf y})d{\bf y}=1,\: i=1,2.
\end{equation}
If
\begin{equation}\label{Compact_Modify}
\int_{D}\int_{D}{\big|}u_{k}({\bf y})-u_{k}({\bf x}){\big|}^2\gamma_{\delta_{k},n_{\delta_{k}}}({\bf x},{\bf y})d{\bf y}d{\bf x}\leq C,
\end{equation}
then $\{u_{k}\}$ is precompact in $L^{2}(D)$. Assume that $u_{k_{j}}\rightarrow u$ in $L^{2}(D)$, then $u\in H^{1}(D)$.
\end{lemma}
\begin{proof}
Let $\rho_{k}({\bf x},{\bf y})=\|{\bf y}-{\bf x}\|_{2}^2\gamma_{\delta_{k}}({\bf x},{\bf y})\chi_{\underline{r}(n_{\delta_{k}})}(\|{\bf y}-{\bf x}\|_{2})$,
by the condition \cref{General_Cond_2D}, \cref{kernel_finite} and \cref{General_kernel}, we know that $\{\rho_{k}\}$ is a sequence of radial functions and for all ${\bf x}\in\Omega$,
\begin{equation}\label{rho_cond_Modify}
     \left \{
     \begin{array}{l}
\rho_{k}({\bf x},{\bf y})\geq 0\:\: \mbox{in} \:\:\mathbb{R}^{2},\\
\displaystyle \vspace{0.1cm}
\lim\limits_{k\rightarrow\infty}\int_{\mathbb{R}^{2}}\rho_{k}({\bf x},{\bf y})d{\bf y}=2,\\
\displaystyle
\lim\limits_{k\rightarrow\infty}\int_{\|{\bf y}-{\bf x}\|_{2}>\varepsilon}\rho_{k}({\bf x},{\bf y})d{\bf y}=0,\: \forall \varepsilon>0.
     \end{array}
     \right.
\end{equation}
The difference between \cref{rho_cond_Modify} and \cref{rho_cond_Stand} is: the second condition of \cref{rho_cond_Modify} takes limit. So we could use similar argument with that in \cite{bourgain2001another,mengesha2012nonlocal,Ponce2004An} to prove the compactness result if \cref{Compact_Modify} is replaced by
\begin{equation*}
\|u_{k} \|_{\underline{r}(n_{\delta_{k}})}^{2}=\int_{D}\int_{D}{\big|}u_{k}({\bf y})-u_{k}({\bf x}){\big|}^2\gamma_{\delta_{k}}({\bf x},{\bf y})\chi_{\underline{r}(n_{\delta_{k}})}(\|{\bf y}-{\bf x}\|_{2})d{\bf y}d{\bf x}\leq C.
\end{equation*}
For example, from the proof of \cite[Theorems 1.2]{Ponce2004An} we know that the essential requirement for $\rho_{k}({\bf x},{\bf y})$ is: for any ${\bf x}\in\Omega$ and fixed $\varepsilon>0$, there is a $k_{0}$ such that
\begin{equation*}
\int_{\|{\bf y}-{\bf x}\|_{2}<\varepsilon}\rho_{k}({\bf x},{\bf y})d{\bf y}\geq 1/2,\: \forall k\geq k_{0},
\end{equation*}
which is assured by the condition \cref{rho_cond_Modify}. Thus we establish the compactness result for the norm $\|\cdot\|_{\underline{r}(n_{\delta_{k}})}$. Then by the relation \cref{equiv_norm} we complete the proof.
\end{proof}

\begin{lemma}\label{lemma:Poincare}
(uniform Poincar\'{e}-type inequality). Assume that
\begin{equation*}
\lim\limits_{k\rightarrow\infty}\int_{\|{\bf y}-{\bf x}\|_{2}\leq\underline{r}(n_{\delta})}|y_{i}-x_{i}|^2\gamma_{\delta}({\bf x},{\bf y})d{\bf y}=1,\: i=1,2.
\end{equation*}
Then there exist $\delta_{0}>0$ and $C>0$ independent of $\delta$
such that $\forall \delta<\delta_{0}$,
\begin{equation}\label{Poincare_low}
\|u 
\|_{0,\Omega}\leq C\|u\|_{\underline{r}(n_{\delta})},
\end{equation}
and thus by \cref{equiv_norm}
\begin{equation}\label{Poincare}
\|u\|_{0,\Omega}\leq C\|u
\|_{\delta,n_{\delta}}.
\end{equation}
\end{lemma}
The inequality \cref{Poincare_low} is similar to \cite[Lemma 3.2]{tian2014asymptotically}, which is a special case of \cite[Proposition 5.3]{mengesha2014bond} for scalar valued functions. To prove \cref{lemma:Poincare}, we just need to prove \cref{Poincare_low}. From \cite{mengesha2014bond} we know that the key point is the compactness result, i.e. \cref{lemma:Compact_Modify} which has been proven.

\begin{theorem}\label{thm:Conv_L2}
Assume ${\big\{} B_{\delta_{k},n_{\delta_{k},{\bf x}}}{\big\}}$ is a family of polygons that satisfy \cref{Subset} and
\begin{equation}\label{under_del}
\lim\limits_{k\rightarrow\infty}\frac{\underline{r}(n_{\delta_{k}})}{\delta_{k}}=1.
\end{equation}
If $f_{\delta_{k}} \rightarrow f_{0}$ in $L^{2}$ sense, the kernels $\{\gamma_{\delta_{k}}\}$ satisfy \cref{kernel_finite} and \cref{General_kernel}, we have
\begin{equation*}
\lim\limits_{k\rightarrow\infty}{\big\|}u_{\delta_{k},n_{\delta_{k}}}-u_{0}{\big\|}_{0,\Omega}=0.
\end{equation*}
\end{theorem}
\begin{proof}
By \cref{under_del}, \cref{kernel_finite} and \cref{General_kernel}, we know that
\begin{equation}\label{General_Cond_2D_new}
\lim_{k\rightarrow\infty}\int_{\|{\bf z}\|_{2}<\underline{r}(n_{\delta_{k}})}z_{i}^2\widetilde{\gamma}_{\delta_{k}}(\|{\bf z}\|_{2})d{\bf z}=\lim_{k\rightarrow\infty}\int_{\|{\bf z}\|_{2}<\delta_{k}}z_{i}^2\widetilde{\gamma}_{\delta_{k}}(\|{\bf z}\|_{2})d{\bf z}=1,\: i=1,2.
\end{equation}
That is the condition \cref{General_Cond_2D} holds.

Take $\delta=\delta_{k}$ in \cref{nonlocal_approx} and multiply $u_{\delta_{k},n_{\delta_{k}}}$ on both sides, we get
\begin{equation*}
{\big\|}u_{\delta_{k},n_{\delta_{k}}}{\big\|}_{\delta_{k},n_{\delta_{k}}}^{2}=-\int_{\Omega}u_{\delta_{k},n_{\delta_{k}}}\mathcal{L}_{\delta_{k},n_{\delta_{k}}}u_{\delta_{k},n_{\delta_{k}}}d{\bf x}
=\int_{\Omega}f_{\delta_{k}}u_{\delta_{k},n_{\delta_{k}}}d{\bf x}.
\end{equation*}
Then by the Cauchy-Schwarz inequality and nonlocal Poincar\'{e}-type inequality \cref{Poincare} in \cref{lemma:Poincare},
\begin{align*}
{\big\|}u_{\delta_{k},n_{\delta_{k}}}{\big\|}_{\delta_{k},n_{\delta_{k}}}^{2}
\leq
{\big\|}f_{\delta_{k}}{\big\|}_{0,\Omega}{\big\|}u_{\delta_{k},n_{\delta_{k}}}{\big\|}_{0,\Omega}\leq C{\big\|}f_{\delta_{k}}{\big\|}_{0,\Omega}{\big\|}u_{\delta_{k},n_{\delta_{k}}}{\big\|}_{\delta_{k},n_{\delta_{k}}}.
\end{align*}
Thus ${\big\|}u_{\delta_{k},n_{\delta_{k}}}{\big\|}_{\delta_{k},n_{\delta_{k}}}\leq C{\big\|}f_{\delta_{k}}{\big\|}_{0,\Omega}$ holds.
This, together with $f_{\delta_{k}} \rightarrow f_{0}$, leads to the following uniform boundedness result
\begin{equation}\label{uniformB}
{\big\|}u_{\delta_{k},n_{\delta_{k}}}{\big\|}_{\delta_{k},n_{\delta_{k}}}\leq C'\|f_{0}\|_{0,\Omega},
\end{equation}
where $C'>0$ is a constant independent of $\delta_k$.
Then by \cref{lemma:Compact_Modify}, we get the convergence of a subsequence of ${\big\{}u_{\delta_{k},n_{\delta_{k}}}{\big\}}$ in $L^{2}$
to a limit point $u'_{0}\in H^{1}_{0}(\Omega)$. For brevity, the same symbol is used to denote the subsequence.

On the other hand, from \cref{Green_Nonlocal}, we know that for all $\varphi\in \mathcal{D}(\Omega)$,
\begin{align*}
\int_{\Omega}\varphi({\bf x})\mathcal{L}_{\delta_{k},n_{\delta_{k}}}u_{\delta_{k},n_{\delta_{k}}}({\bf x})d{\bf x}&
=\int_{\Omega}u_{\delta_{k},n_{\delta_{k}}}({\bf x})\left(\sum_{i=1}^{2}\sigma_{n_{\delta_{k}},i}({\bf x})\partial_{ii}^{2}\varphi({\bf x})\right)d{\bf x}\\
&+\mathcal{O}{\big(}\delta_{k}^2{\big)}{\big|}\varphi{\big|}_{4,\infty}{\big\|}u_{\delta_{k},n_{\delta_{k}}}{\big\|}_{L^1(\Omega)}.
\end{align*}
Thus by the uniform boundedness result \cref{uniformB} and the uniform Poincar\'{e}-type inequality \cref{Poincare} in \cref{lemma:Poincare}, the following formula holds for all $\varphi\in \mathcal{D}(\Omega)$,
\begin{equation}\label{delta_form}
\int_{\Omega}u_{\delta_{k},n_{\delta_{k}}}({\bf x}){\bigg(}\sum_{i=1}^{2}\sigma_{n_{\delta_{k}},i}({\bf x})\partial_{ii}^{2}\varphi({\bf x}){\bigg)}d{\bf x}+\mathcal{O}{\big(}\delta_{k}^2{\big)}=-\int_{\Omega}f_{\delta_{k}}({\bf x})\varphi({\bf x})d{\bf x}.
\end{equation}
By the condition \cref{Subset} and the definition of $\underline{r}(n_{\delta_{k}})$, we know that for all ${\bf x}\in\Omega$
\begin{equation*}
\int_{\|{\bf z}\|_{2}<\underline{r}(n_{\delta_{k}})}z_{i}^2\widetilde{\gamma}_{\delta_{k}}\left(\|{\bf z}\|_{2}\right)d{\bf z}\leq \sigma_{n_{\delta_{k}},i}({\bf x})\leq\int_{\|{\bf z}\|_{2}<\delta_{k}}z_{i}^2\widetilde{\gamma}_{\delta_{k}}\left(\|{\bf z}\|_{2}\right)d{\bf z},\: i=1,2,
\end{equation*}
then by \cref{General_Cond_2D_new} and the squeeze theorem, the following limit result holds
\begin{equation}\label{sigma_lim}
\lim_{k\rightarrow\infty}\sigma_{n_{\delta_{k}},i}({\bf x})=1,\: i=1,2,\: \forall {\bf x}\in\Omega.
\end{equation}
Taking limit in \cref{delta_form}, it is shown that $u_{0}'$ satisfies \cref{local_weak},
the proof is complete.
\end{proof}

We note from the proof that, in the above theorem, the condition on  the
$L^2$ convergence of $f_{\delta_{k}}$ to $ f_{0}$ can be relaxed, we refer to for related discussions in \cite{tian2020asymptotically}. Meanwhile, the condition on the polygonal approximations to the interaction neighborhood for \cref{thm:Conv_L2} is stated as a very general condition. In this paper, we also establish two other theorems which impose more specialized conditions. However, they  remain applicable in many circumstances.
\begin{theorem}\label{thm:Conv_L2_Area}
Assume that ${\big\{} B_{\delta_{k},n_{\delta_{k},{\bf x}}}{\big\}}$ is a family of convex polygons that satisfy \cref{Subset} and
\begin{equation}\label{Area_cond}
\inf_{{\bf x}\in\Omega}\frac{{\big|}B_{\delta_{k},n_{\delta_{k},{\bf x}}}{\big|}}{{\big|}B_{\delta_{k}}({\bf x}){\big|}}\rightarrow 1\:\: \mbox{as}\:\: k\rightarrow\infty.
\end{equation}
If $f_{\delta_{k}} \rightarrow f_{0}$ in $L^{2}$ sense, the kernels  $\{\gamma_{\delta_{k}}\}$ satisfy \cref{kernel_finite} and \cref{General_kernel}, the same conclusion as \cref{thm:Conv_L2} holds.
\end{theorem}
\begin{proof}
According to \cref{thm:Conv_L2}, we need to show \cref{under_del} by \cref{Area_cond} and the `convexity condition'. We prove this by contradiction. If \cref{under_del} does not hold, then there exists a constant $\varepsilon_{0}\in (0,1/2)$ such that for all $k$, $\underline{r}(n_{\delta_{k}})/\delta_{k}\leq 1-2\varepsilon_{0}$. From the definition of $\underline{r}(n_{\delta_{k}})$, for any $k$ there exists a point ${\bf x}_{0}\in\Omega$ such that $r_{\delta_{k},{\bf x}_{0}}/\delta_{k}\leq 1-\varepsilon_{0}$.
Assume that the inscribed ball of $B_{\delta_{k},n_{\delta_{k},{\bf x}_{0}}}$ is tangent to the polygon at point $P$ which is on the side $\overline{A_{1}A_{2}}$. The side $\overline{A_{1}A_{2}}$ must be a part of a chord, which is denoted by $\overline{A_{3}A_{4}}$, see (a) of \cref{Fig:Area_cond}. Denote the cap corresponding to the chord $\overline{A_{3}A_{4}}$ by $C_{\rm grey}$, which is the cap with grey color in (a) of \cref{Fig:Area_cond}. Since the polygon $B_{\delta_{k},n_{\delta_{k},{\bf x}_{0}}}$ is convex,
$C_{\rm grey}\subset B_{\delta_{k}}({\bf x}_{0})\setminus B_{\delta_{k},n_{\delta_{k},{\bf x}_{0}}}$. The ball, polygon and cap are mapped to their counterparts in $B_{1}({\bf 0})$ by the mapping \cref{Affine_map}, see (b) of \cref{Fig:Area_cond} where the image of $C_{\rm grey}$ is denoted by $\widehat{C}_{\rm grey}$. Thus,
\begin{equation}\label{conf_rela}
\widehat{C}_{\rm grey}\subset B_{1}({\bf 0})\setminus B_{1,n_{\delta_{k},{\bf x}_{0}}}({\bf 0}).
\end{equation}
\begin{figure}[tbhp]
\vspace{-0.5cm}
\centering
\subfloat[convex polygon in $B_{\delta_{k}}({\bf x}_{0})$]{
\includegraphics[width=4cm]{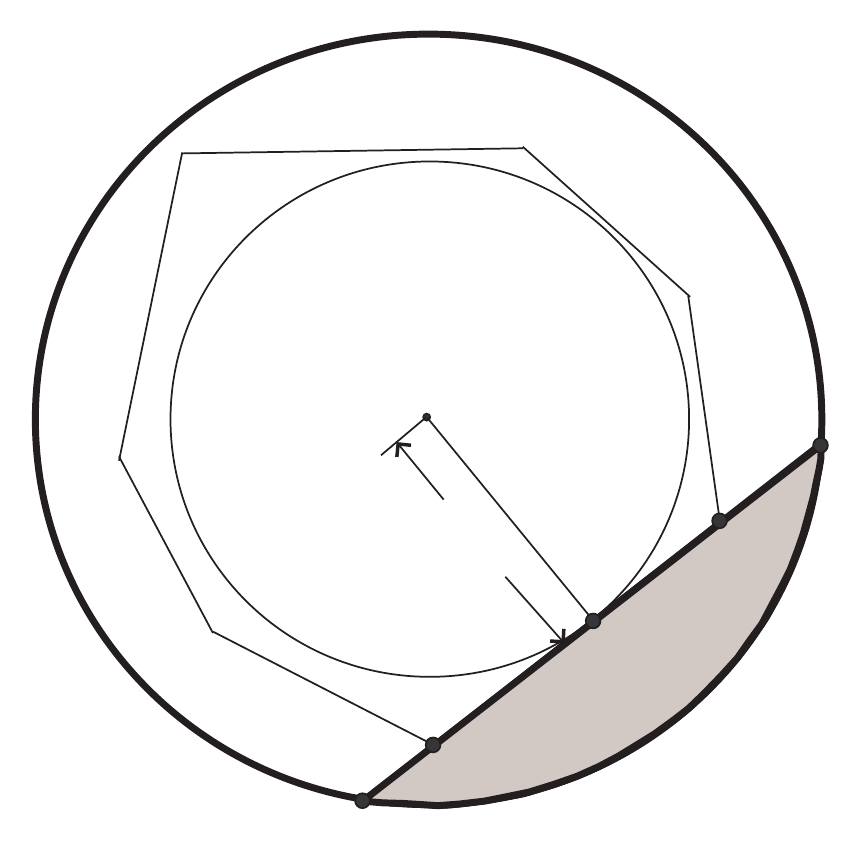}
\put(-65,60){${\bf x}_{0}$}\put(-67,39){$r_{\delta_{k},{\bf x}_{0}}$}
\put(-57,8){$A_{1}$}\put(-23,34){$A_{2}$}\put(-36,22){$P$}
\put(-72,-3){$A_{3}$}\put(-5,50){$A_{4}$}}\hspace{1cm}
\subfloat[convex polygon in $B_{1}({\bf 0})$]{
\includegraphics[width=5cm]{Convex_Poly_Contradiction}
\put(-79,73){$O$}\put(-95,48){$r_{\delta_{k},{\bf x}_{0}}/\delta_{k}$}
\put(-70,9){$\widehat{A}_{1}$}
\put(-28,41){$\widehat{A}_{2}$}
\put(-47,25){$\widehat{P}$}
\put(-91,-5){$\widehat{A}_{3}$}\put(-5,63){$\widehat{A}_{4}$} }
\caption{Convex polygon in the ball $B_{\delta_{k}}({\bf x}_{0})$ and its image by the mapping \cref{Affine_map}}\label{Fig:Area_cond}
\end{figure}
By simple calculation, we have
\begin{equation*}
{\Big|}\widehat{C}_{\rm grey}{\Big|}=\theta_{n_{\delta_{k},{\bf x}_{0}}}-\sin\left(\theta_{n_{\delta_{k},{\bf x}_{0}}}\right),
\end{equation*}
with
\begin{equation*}
\theta_{n_{\delta_{k},{\bf x}_{0}}}=2\arccos\left(r_{\delta_{k},{\bf x}_{0}}/\delta_{k}\right)\geq 2\arccos(1-\varepsilon_{0}),
\end{equation*}
which leads to
\begin{equation}\label{conf_area}
{\Big|}\widehat{C}_{\rm grey}{\Big|}\geq 2\arccos(1-\varepsilon_{0})-\sin\left(2\arccos(1-\varepsilon_{0})\right)>0.
\end{equation}
By \cref{conf_rela} and \cref{conf_area} we know that for any $k$ there exists a point ${\bf x}_{0}\in\Omega$ such that
\begin{equation*}
\frac{{\big|}B_{\delta_{k},n_{\delta_{k},{\bf x}_{0}}}{\big|}}{{\big|}B_{\delta_{k}}({\bf x}_{0}){\big|}}=\frac{{\big|}B_{1,n_{\delta_{k},{\bf x}_{0}}}{\big|}}{{\big|}B_{1}({\bf 0}){\big|}}\leq 1-\widehat{C}_{\rm grey}\leq 1-2\arccos(1-\varepsilon_{0})+\sin\left(2\arccos(1-\varepsilon_{0})\right)<1
\end{equation*}
which contradicts the condition \cref{Area_cond}.
\end{proof}
\subsubsection{Local limit as $n_\delta\to \infty$ for quasi-uniform inscribed polygons}\label{subsec:ndel_inf}
If the family of polygons possesses both `inscribed' and `quasi-uniform' properties, the condition \cref{Area_cond} could be further simplified.
\begin{definition}
A family of polygons $\{B_{\delta,n_{\delta,{\bf x}}}\}$ is called quasi-uniform if there exist two constants $C_{1}$ and $C_{2}>0$ such that $\forall\delta>0$, the following two bounds hold
\begin{equation*}
\sup_{{\bf x}\in\Omega}\frac{H^{\rm max}_{n_{\delta,{\bf x}}}}{H^{\rm min}_{n_{\delta,{\bf x}}}}\leq C_{1},
\end{equation*}
where $H^{\rm max}_{n_{\delta,{\bf x}}}$ and $H^{\rm min}_{n_{\delta,{\bf x}}}$ stand for the lengths of the longest and shortest sides of the polygon $B_{\delta,n_{\delta,{\bf x}}}$, respectively,
and
\begin{equation*}
\sup_{{\bf x}\in\Omega}\frac{\delta}{r_{\delta,{\bf x}}}\leq C_{2}.
\end{equation*}
\end{definition}
Then for a quasi-uniform family of polygons, there exists a constant $C>0$ such that for all $\delta>0$, $n_{\delta}\leq C n_{\delta,\inf}$ holds.
\begin{theorem}\label{thm:Conv_L2_Sides}
Assume that ${\big\{} B_{\delta_{k},n_{\delta_{k},{\bf x}}}{\big\}}$ is a quasi-uniform family of inscribed polygons. The kernels $\gamma_{\delta_{k}}$ satisfy \cref{kernel_finite} and \cref{General_kernel}.
If $f_{\delta_{k}} \rightarrow f_{0}$ in $L^{2}$, and
\begin{equation}\label{sides_cond}
\lim\limits_{k\rightarrow\infty}n_{\delta_{k}}=\infty,
\end{equation}
then we have $\lim\limits_{k\rightarrow\infty}{\big\|}u_{\delta_{k},n_{\delta_{k}}}-u_{0}{\big\|}_{0,\Omega}=0$.
\end{theorem}
\begin{proof}
By \cref{sides_cond} and the quasi-uniformity of ${\big\{} B_{\delta_{k},n_{\delta_{k},{\bf x}}}{\big\}}$, we know that
\begin{equation}\label{sidestoinf}
n_{\delta_{k},{\bf x}}\geq n_{\delta_{k},\inf}\geq \frac{1}{C} n_{\delta_{k}}\rightarrow\infty,\: \forall {\bf x}\in\Omega,
\end{equation}
and thus \cref{Area_cond} holds. Then by \cref{thm:Conv_L2_Area} we complete the proof.
\end{proof}

\begin{remark}\label{Remark:Regular}
As a typical example, for a family of regular inscribed polygons, if \cref{sides_cond} holds, we have $\underline{r}(n_{\delta_{k}})=\delta_{k}\cos\frac{\pi}{n_{\delta_{k}}}$.
Thus \cref{under_del} holds obviously.
\end{remark}
\begin{remark}
C. Vollmann \cite{vollmann2019nonlocal} extended the convergence result for kernels with Euclidean ball as interaction neighborhood illustrated in \cref{subsec:Limit_Euclid_ball} to that with balls induced by  any norm, but satisfy an analogous scaling condition. For example, if the $\|\cdot\|_{\infty}$ ball $B_{\delta,\infty}({\bf x})$ (square) is used as interaction neighborhood, the convergence of corresponding nonlocal solution to the local solution $u_{0}$ is well established if the second condition of \cref{General_kernel} is replaced by
\begin{equation}\label{Sec_moment_inf}
\int_{B_{\delta,\infty}({\bf 0})}z_{i}^{2}\widetilde{\gamma}_{\delta,\infty}(\|{\bf z}\|_{2})d{\bf z}=1,\: i=1,2.
\end{equation}
The nonlocal problems with a general norm (other than Euclidean norm) induced balls could be regarded as new models parallel to that use Euclidean ball as interaction neighborhood.
Such a conclusion are consistent to our findings here. In fact, the normalization condition \cref{Sec_moment_inf} is equivalent to the use of a
re-scaled nonlocal operator corresponding to  $(4/C_4)\mathcal{L}_{\delta,n_{\delta}}$, which has a consistent local limit as shown in \cref{conv_scale}.
\end{remark}

\subsection{Higher dimensional cases}\label{subsec:highd}
So far most of the discussions have been carried out for two dimensional case. Now we extend the established results to general high dimensional
case. In fact, \cref{thm:Conv_L2} and \cref{thm:Conv_L2_Area} are stated and indeed valid in any high dimensional case subject to the corresponding conditions on the kernels and with the polygonal approximations of the two dimensional ball changed to polyhedral approximations of high dimensional balls.
Meanwhile, modifications and extensions of \cref{thm:Conv_L2_Sides} and \cref{Remark:Regular} in high dimensional settings need to be further investigated.

\subsection{Examples of kernels}\label{subsec:kernel}
Some popular kernels satisfy the conditions \cref{kernel_finite} and \cref{General_kernel}. Here we list two special types of them for the general $d$ dimensional setting, however, the classification is not strict. For more discussions on the effects of the kernels on the nonlocal models, we refer to \cite{du2019nonlocal,seleson2011role}.

\noindent
\textit{Type} 1. Integrable kernel, for example,

\textit{Type} 1.1
Constant kernel:
\begin{equation*}
\gamma(t)=\frac{d(d+2)}{w_{d}},\:\: \mbox{for}\:\: 0\leq t< 1,
\end{equation*}
where $w_{d}$ is the surface area of the unit sphere in $\mathbb{R}^d$.

\textit{Type} 1.2
Linear kernel:
\begin{equation*}
\gamma(t)=\frac{d(d+2)(d+3)}{w_{d}}(1-t),\:\: \mbox{for}\:\: 0\leq t< 1.
\end{equation*}

\textit{Type} 1.3  Gaussian-like exponential kernel:
\begin{equation*}
\gamma(t)=\frac{d}{C_{e}w_{d}}e^{-t^2},\:\: \mbox{for}\:\: 0\leq t<1,\:\: \mbox{with} \:\: C_{e}=\int_{0}^{1}\tau^{d+1}e^{-\tau^2}d\tau.
\end{equation*}

\noindent
\textit{Type} 2. Singular (at the origin) kernel:
\begin{equation}\label{singular_kernel}
\gamma(t)=\frac{d(d+2-s)}{w_{d}}t^{-s}\:\: \mbox{for}\:\: 0< t< 1,
\end{equation}
which stands for different kernels for different values of $s$. For $s\in(d,d+2)$ it is  the fractional kernel, for $s=1$ it is the often used peridynamics kernel \cite{silling2005meshfree}, for $s<d$ it is in fact integrable.

\section{Verification of the theory, numerical results and discussions}\label{sec:numer}
\subsection{Verification of the theory}
By the proof of \cref{thm:Conv_L2} we know that whether the formula \cref{sigma_lim} holds or not is crucial for the convergence result of \cref{thm:Conv_L2}, \cref{thm:Conv_L2_Area} and \cref{thm:Conv_L2_Sides}. Here we give two examples, which differ in the choices of the kernels, to compute explicitly the corresponding $\sigma_{n,i}$, using the family of inscribed regular even-sided polygons ${\big\{} B_{\delta,n}{\big\}}$.

The first example uses the constant kernel \cref{const_2d}, the following result holds by simple calculation,
\begin{equation*}
\sigma_{n,i}=\frac{\sin(2\pi/n)}{2\pi/n}\cdot\frac{2+\cos(2\pi/n)}{3},\: i=1,2,
\end{equation*}
which is also related to \cref{conv_wrong} since $\mathcal{L}_{\delta,n_{\delta}}q({\bf x})=C_{n}=2(\sigma_{n,1}+\sigma_{n,2})$.

The second example uses the so-called peridynamics kernel, that is \cref{singular_kernel} with $d=2$ and $s=1$, it holds that
\begin{equation*}
\sigma_{n,i}=\frac{n}{\pi}{\big(}\cos(\pi/n){\big)}^{3}\int_{0}^{\pi/n}\frac{d\theta}{(\cos\theta)^{3}},\: i=1,2.
\end{equation*}
Obviously, for the two examples above, we have
\begin{equation*}
\sigma_{n,i}<1,\: i=1,2,\: \forall n,
\end{equation*}
which verifies \cref{thm:Notconv}, see \cref{conv_wrong}.
Moreover,
\begin{equation}\label{sigma_lim_ex}
\lim\limits_{n\rightarrow\infty}\sigma_{n,i}=1,\: i=1,2,
\end{equation}
which confirms the assertion \cref{sigma_lim} under the condition \cref{sides_cond}. This also partially verifies \cref{thm:Conv_L2_Sides}. To fully verify \cref{thm:Conv_L2_Sides}, in the following subsection we use numerical solutions computed by taking the mesh size $h$ to be of higher-order of the horizon parameter $\delta$ as $\delta\to 0$ (resulting in numerically constructed polygonal interaction neighborhoods with $n_\delta\to \infty$) to illustrate that \cref{thm:Conv_L2_Sides} holds numerically.

\subsection{Numerical results}
\begin{example}\label{Example:Continuous}
We consider the nonlocal problem \cref{nonlocal_diffusion} and \cref{nonlocal_approx} on the domain $\Omega=(0,1)^{2}$ with the constant kernel \cref{const_2d}, where $\mathcal{L}_{\delta,n_{\delta}}$ is defined by the strategy $\sharp=nocaps$ to approximate balls. Note that, for ${\bf x}\in\Omega$, $\mathcal{L}_{\delta}u_{\delta}({\bf x})=\Delta u_{\delta}({\bf x})$ for $\delta\geq 0$ and polynomials $u_{\delta}({\bf x})$ with order up to three \cite{d2021cookbook,vollmann2019nonlocal}. As in \cite{d2021cookbook} the manufactured solution $u_{\delta}({\bf x})=x_{1}^2x_{2}+x_{2}^{2}$ of \cref{nonlocal_diffusion} is used to obtain the right hand side function $f_{0}({\bf x})=f_{\delta}({\bf x})=-\Delta u_{\delta}=-2(x_{2}+1)$ for ${\bf x}\in\Omega$.
\end{example}

In fact the solution of corresponding local problem \cref{local_weak} is
$u_{0}({\bf x})=u_{\delta}({\bf x})|_{\Omega}=x_{1}^2x_{2}+x_{2}^{2}$. This means that the nonlocal solution of \cref{nonlocal_diffusion} converges to the local solution of \cref{local_weak}.

To verify the conditional AC property of the polygonal approximation, the nonlocal problem \cref{nonlocal_approx} is discretized by the conforming DG method proposed in \cite{du2019conforming} with quasi-uniform triangulation.

Since $\mathcal{L}_{\delta,n_{\delta}}$ is defined by the strategy $\sharp={\rm nocaps}$, we denote the corresponding conforming DG solution of \cref{nonlocal_approx}  as $u_{\delta,n_{\delta}}^{\rm D\!, nocaps}$. We set $h=\mathcal{O}{\big(}\delta^{\beta}{\big)}$ with $\beta>1$. On one hand, for every $\delta$ mesh size $h$ is relatively small enough so the error caused by the conforming DG is also small enough, that is $u_{\delta,n_{\delta}}^{\rm D\!, nocaps} \approx u_{\delta,n_{\delta}}$. On the other hand, $\delta/h\rightarrow\infty$ holds (leading to $n_\delta\to \infty$) as $\delta\rightarrow 0$. From \cref{Fig:L2err}, we find positive convergence rates that are dependent on $\beta$. More detailed estimates of the error $u_{\delta,n_{\delta}}^{\rm D\!, nocaps}-u_{\delta,n_{\delta}}$ will be derived in a subsequent work.
\begin{figure}[tbhp]
\centering
\includegraphics[width=6.5cm]{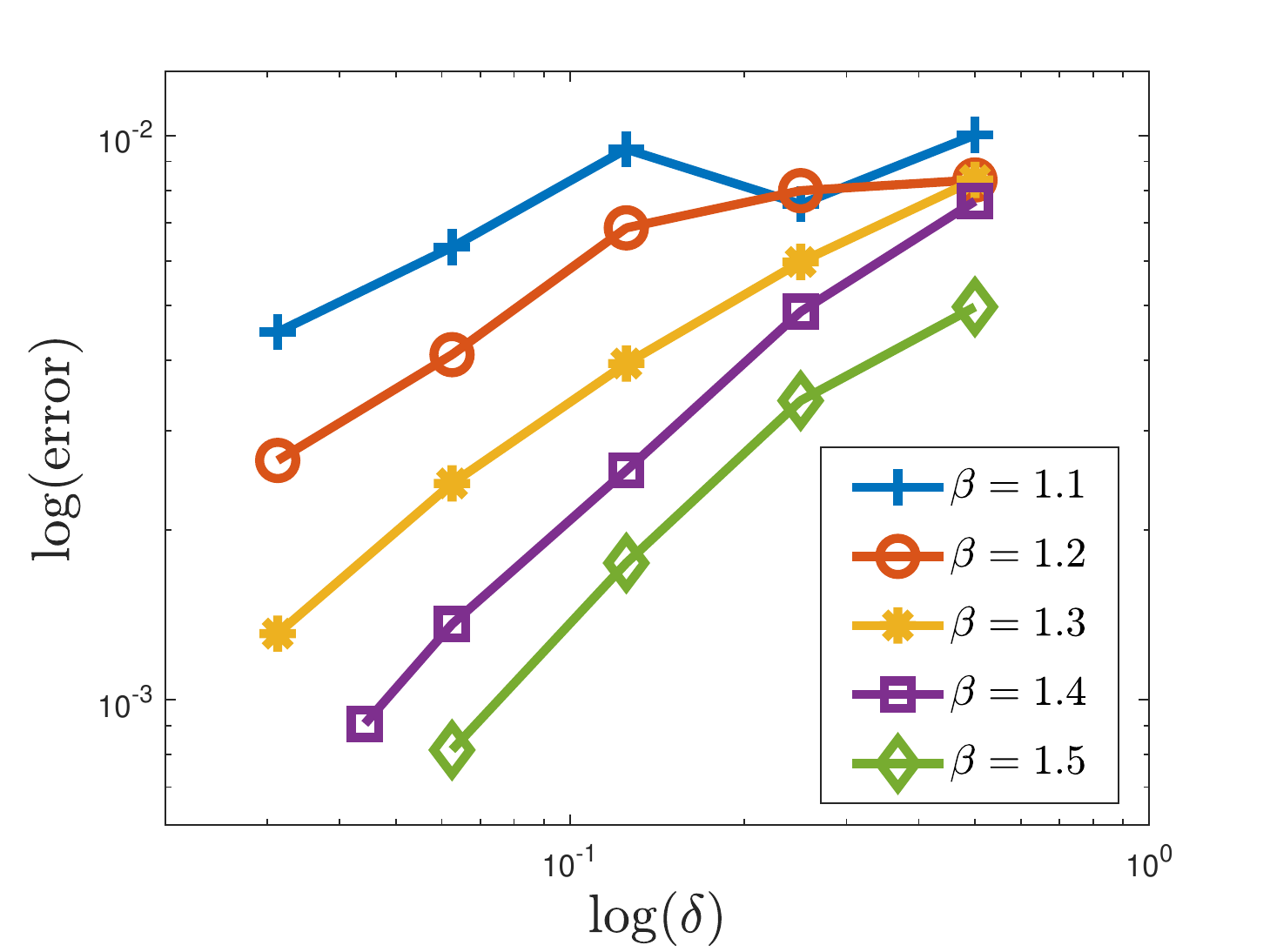}
\caption{$L^{2}$ error between the conforming DG solution of \cref{nonlocal_approx} and the exact local solution $u_{0}$}\label{Fig:L2err}
\end{figure}

\subsection{Discussions}
Now, we extend the results above to the other cases in \cref{appro_ball} which are used to construct the corresponding polygons and thus define the operator $\mathcal{L}_{\delta,n_{\delta}}$. The strategy $\sharp=\rm regular$ is a special case of $\sharp=\rm nocaps$. The strategy $\sharp=\rm approxcaps$ also uses inscribed polygons to approximate the exact ball, the conclusion is similar to that of $\sharp=\rm nocaps$. To be specific, for bounded $\delta/h$ and quasi-uniform triangulation, if the number of sub-triangles that approximate the caps is also bounded, then by \cref{thm:Notconv} $u_{\delta,n_{\delta}}$ does not converge to the exact local solution as $\delta\rightarrow 0$. On the other hand, if the number of sub-triangles tends to infinity as $\delta\rightarrow 0$, thus \cref{sides_cond} holds, then  by \cref{thm:Conv_L2_Sides} we know $u_{\delta,n_{\delta}}$ converges to the exact local solution as $\delta\rightarrow 0$. For the strategy $\sharp=\rm 3vertices$: if $\delta/h$ is uniformly bounded as $\delta\rightarrow 0$, then by \cref{thm:Notconv} the non-convergence is established. If $\delta/h$ tends to infinity as $\delta\rightarrow 0$, we can not use \cref{thm:Conv_L2_Area} or \cref{thm:Conv_L2_Sides} since the polygons are not always convex. However, we could use \cref{thm:Conv_L2} to prove that $u_{\delta,n_{\delta}}$ converges to $u_{0}$ as $\delta\rightarrow 0$. In fact, for quasi-uniform triangulation we have
\begin{equation*}
\lim\limits_{\delta\rightarrow 0}\frac{\underline{r}(n_{\delta})}{\delta}\geq\lim\limits_{\delta\rightarrow 0}\frac{\delta-h}{\delta}=1,
\end{equation*}
under the condition that $\delta/h$ tends to infinity as $\delta\rightarrow 0$, which means \cref{under_del} holds.

For the rest types of polygonal approximations of Euclidean balls in \cref{appro_ball}, i.e.
\begin{equation}\label{sharp_rest}
\sharp\in\big\{\rm 123vertices, 23vertices, barycenter\big\},
\end{equation}
the corresponding families of polygons are no longer contained in corresponding Euclidean balls, that is,  the condition \cref{Subset} does not hold. However, the analysis is similar to that derived in this paper, with consideration for suitable extension rules of the original kernel to the outside of the interaction neighborhood.

\section{Conclusion}\label{sec:conclusion}
We discussed in this paper some new nonlocal operators which, on the continuum level, are approximations of the nonlocal operators with the Euclidean balls as interaction neighborhoods. They are defined through polygons that approximate the interaction neighborhoods of the original operators. It is well known that the original nonlocal operators converge to the local operator, at the same time the convergence for the nonlocal solutions to the local one is also well established. However, as shown in \cref{subsec:lossAC}, the new nonlocal operators may not converge to the local operator as $\delta$ vanishes if the number of sides of the polygons, $n_{\delta}$, is uniformly bounded. This phenomenon is interpreted as a lost of AC property for the polygonal approximation. We proved that the new nonlocal solutions converge to the local solution under suitable conditions, which means the polygonal approximation of radially symmetric interaction neighborhood possesses conditional AC property.

The current study focuses on the convergence of the nonlocal solutions $u_{\delta,n_{\delta}}$ to the local one. However, when some specific numerical schemes are used, one interesting question is how the convergence order depends on parameters with respect to the vanishing parameter $\delta$. This issue is briefly touched upon in \cref{Fig:L2err} where the convergence orders are shown to be dependent on $\beta$. More details will be discussed in a subsequent work.
Naturally, extensions to higher dimensions and to nonlinear and time-dependent problems, as well as studies on the coupling with suitable quadrature schemes, are also questions that remain to be further studied.

\appendix
\section{The limiting behavior of the norm}\label{sec:appendix}
\begin{theorem}\label{thm:Conv_Stand}
Assume that the kernels $\{\gamma_{\delta_{k}}\}$ satisfy the conditions \cref{kernel_finite} and \cref{General_kernel}. We have the following pointwise limit
\begin{equation}\label{Formula_Conv_Stand}
\lim_{k\rightarrow\infty}\int_{D}\int_{D}{\big|}u({\bf y})-u({\bf x}){\big|}^2\gamma_{\delta_{k}}({\bf x},{\bf y})d{\bf y}d{\bf x}= \int_{D}|\nabla u({\bf x})|^{2}d{\bf x}.
\end{equation}
\end{theorem}
\begin{proof}
Let $\rho_{k}({\bf x},{\bf y})=\|{\bf y}-{\bf x}\|_{2}^2\gamma_{\delta_{k}}({\bf x},{\bf y})$, then $\{\rho_{k}\}$ is a sequence of radial
functions and satisfies the conditions \cref{rho_cond_Stand}. From \cite[Corollary 1]{bourgain2001another} or \cite[Theorem 2]{brezis2002recognize}, we know
\begin{equation*}
\lim_{k\rightarrow\infty}\int_{D}\int_{D}\frac{{\big|}u({\bf y})-u({\bf x}){\big|}^2}{\|{\bf y}-{\bf x}\|_{2}^2}\rho_{k}({\bf x},{\bf y})d{\bf y}d{\bf x}= \int_{D}|\nabla u({\bf x})|^{2}d{\bf x},
\end{equation*}
which is in fact \cref{Formula_Conv_Stand}, we then complete the proof.
\end{proof}
\begin{theorem}\label{thm:Conv_Modify}
Under the condition of \cref{thm:Conv_Stand}, and assume that the counterpart in $d$ dimensional case of \cref{General_Cond_2D} holds, we have the following pointwise limit
\begin{equation*}
\lim_{k\rightarrow\infty}\int_{D}\int_{D}{\big|}u({\bf y})-u({\bf x}){\big|}^2\gamma_{\delta_{k}}({\bf x},{\bf y})\chi_{\underline{r}(n_{\delta_{k}})}{\big(}\|{\bf y}-{\bf x}\|_{2}{\big)}d{\bf y}d{\bf x}
= \int_{D}|\nabla u({\bf x})|^{2}d{\bf x}.
\end{equation*}
\end{theorem}
\begin{proof}
From the proof of \cite[Theorem 2]{bourgain2001another}, we just need to give the proof for $u\in C^{2}(\bar{D})$. In this situation,
\begin{equation*}
{\big|}u({\bf y})-u({\bf x}){\big|}={\big|}({\bf y}-{\bf x})\cdot\nabla u({\bf x}){\big|}+o{\big(}\|{\bf y}-{\bf x}\|_{2}^{2}{\big)}.
\end{equation*}
For each fixed ${\bf x}\in\Omega$,
\begin{align*}
&\int_{D}{\big|}u({\bf y})-u({\bf x}){\big|}^2\gamma_{\delta_{k}}({\bf x},{\bf y})\chi_{\underline{r}(n_{\delta_{k}})}{\big(}\|{\bf y}-{\bf x}\|_{2}{\big)}d{\bf y}\\
&=\int_{\|{\bf y}-{\bf x}\|_{2}<\mbox{dist}({\bf x},\partial D)}\mbox{integrand}\:d{\bf y}+\int_{\|{\bf y}-{\bf x}\|_{2}\geq\mbox{dist}({\bf x},\partial D)}\mbox{integrand}\:d{\bf y}.
\end{align*}
The second integral tends to $0$ as $k\rightarrow \infty$. Set $R=\mbox{dist}({\bf x},\partial D)$, when $\underline{r}(n_{\delta_{k}})\leq R$,
\begin{align*}
&\int_{\|{\bf y}-{\bf x}\|_{2}<R}{\big|}u({\bf y})-u({\bf x}){\big|}^2\gamma_{\delta_{k}}({\bf x},{\bf y})\chi_{\underline{r}(n_{\delta_{k}})}{\big(}\|{\bf y}-{\bf x}\|_{2}{\big)}d{\bf y}\\
&=\int_{\|{\bf z}\|_{2}<\underline{r}(n_{\delta_{k}})}{\Big(}{\big|}{\bf z}\cdot\nabla u({\bf x}){\big|}^{2}+o{\big(}r^{2}{\big)}{\Big)}\widetilde{\gamma}_{\delta_{k}}(\|{\bf z}\|_{2})d{\bf z}\\
&=\int_{\|{\bf z}\|_{2}<\underline{r}(n_{\delta_{k}})}{\big|}{\bf z}\cdot\nabla u({\bf x}){\big|}^{2}\widetilde{\gamma}_{\delta_{k}}(\|{\bf z}\|_{2})d{\bf z}
+\int_{0}^{\underline{r}(n_{\delta_{k}})}\rho_{k}(r)\int_{\|{\bf z}\|_{2}=r}o{\big(}r^{2}{\big)}d{\bf S}_{d-1}dr\\
&=\sum_{i=1}^{d}{\big(}\partial_{i}u({\bf x}){\big)}^{2}\int_{\|{\bf z}\|_{2}<\underline{r}(n_{\delta_{k}})}z_{i}^2\widetilde{\gamma}_{\delta_{k}}(\|{\bf z}\|_{2})d{\bf z}+o\left(\int_{0}^{\underline{r}(n_{\delta_{k}})}r^{d+1}\widetilde{\gamma}_{\delta_{k}}(r)dr \right).
\end{align*}
From the second condition of \cref{General_kernel} and \cref{General_Cond_2D}, we know that
\begin{equation*}
\lim_{k\rightarrow\infty}\int_{\|{\bf z}\|_{2}<\underline{r}(n_{\delta_{k}})}z_{i}^2\widetilde{\gamma}_{\delta_{k}}(\|{\bf z}\|_{2})d{\bf z}=\lim_{k\rightarrow\infty}\int_{\|{\bf z}\|_{2}<\delta_{k}}z_{i}^2\widetilde{\gamma}_{\delta_{k}}(\|{\bf z}\|_{2})d{\bf z}=1,\: i=1,2,\cdots,d.
\end{equation*}
Thus
\begin{equation*}
\lim_{k\rightarrow\infty}\int_{D}{\big|}u({\bf y})-u({\bf x}){\big|}^2\gamma_{\delta_{k}}({\bf x},{\bf y})\chi_{\underline{r}(n_{\delta_{k}})}{\big(}\|{\bf y}-{\bf x}\|_{2}{\big)}d{\bf y}
=|\nabla u({\bf x})|^{2}.
\end{equation*}
\end{proof}

\begin{theorem}\label{thm:Conv_Oper}
Under the condition of \cref{thm:Conv_Modify}, we have the pointwise limit
\begin{equation}\label{Cond_Conv}
\lim_{k\rightarrow\infty}\int_{D}\int_{D}{\big|}u({\bf y})-u({\bf x}){\big|}^2\gamma_{\delta_{k},n_{\delta_{k}}}({\bf x},{\bf y})d{\bf y}d{\bf x}= \int_{D}|\nabla u({\bf x})|^{2}d{\bf x}.
\end{equation}
This means $\lim\limits_{k\rightarrow\infty}\|u({\bf x})\|_{\delta_{k},n_{\delta_{k}}}^{2}= \int_{D}|\nabla u({\bf x})|^{2}d{\bf x}$.
\end{theorem}
\begin{proof}
By \cref{equiv_norm}, that is $\forall u\in V(\widehat{\Omega}_{\delta_{k}})$,
$\|u({\bf x})\|_{\underline{r}(n_{\delta_{k}})}\leq\|u({\bf x})\|_{\delta_{k},n_{\delta_{k}}}\leq\|u({\bf x})\|_{\delta_{k}}$.
Together with \cref{thm:Conv_Stand}, \cref{thm:Conv_Modify}, and the squeeze theorem, we get \cref{Cond_Conv}.
\end{proof}

\bibliographystyle{siamplain}
\bibliography{references}
\end{document}